\documentclass[a4paper,psamsfonts,reqno]{amsart}
\usepackage{amssymb,amscd,amsmath}

\usepackage{upref}
\usepackage[american]{babel}
\usepackage[all]{xy}
\usepackage{fancyhdr}
\usepackage{enumerate}
\usepackage[active]{srcltx}

\theoremstyle{plain}
\newtheorem{theo}{Theorem}[section]

\newtheorem{prop}[theo]{Proposition}
\newtheorem{lem}[theo]{Lemma}
\newtheorem{coro}[theo]{Corollary}

\newtheorem*{Cauchon}{Cauchon's Theorem}

 \theoremstyle{definition}

 \newtheorem{ex}[theo]{Example}

 \theoremstyle{remark}

 \newtheorem{rem}[theo]{Remark}

\numberwithin{equation}{section}

\newcommand{\freealgebra}[2]{{#1} \langle {#2}\rangle}

\newcommand{\integers}{\mathbb{Z}}

\DeclareMathOperator{\gr}{gr}

\begin{document}

\title[Free symmetric algebras in division rings]{Free symmetric  algebras in division rings
generated by  enveloping algebras of Lie algebras}

\author[V. O. Ferreira]{Vitor O. Ferreira}
\address{Department of Mathematics - IME, University of S\~ao Paulo,
Caixa Postal 66281, S\~ao Paulo, SP, 05314-970, Brazil}
\email{vofer@ime.usp.br}
\thanks{The first and second authors were partially supported by
Fapesp-Brazil Proc.~2009/52665-0.}

\author[J. Z. Gon\c calves]{Jairo Z. Gon\c calves}
\address{Department of Mathematics - IME, University of S\~ao Paulo,
Caixa Postal 66281, S\~ao Paulo, SP, 05314-970, Brazil}
\email{jz.goncalves@usp.br}
\thanks{The second author was partially supported by CNPq-Brazil Grant 301.320/2011-0.}

\author[J. S\'{a}nchez]{Javier S\'{a}nchez}
\address{Department of Mathematics - IME, University of S\~ao Paulo,
Caixa Postal 66281, S\~ao Paulo, SP, 05314-970, Brazil}
\email{jsanchez@ime.usp.br}
\thanks{The third  author was partially supported by DGI MINECO
MTM2011-28992-C02-01, by FEDER UNAB10-4E-378 ``Una manera de hacer
Europa'', and by the Comissionat per Universitats i Recerca de la
Generalitat de Catalunya.}

\subjclass[2010]{Primary 16K40, 16S36, 16W10; Secondary 16S10, 16S30}

\keywords{Infinite dimensional division rings, division
rings with involution, free associative algebras, enveloping algebras of Lie algebras}

\date{06 June 2014}

\begin{abstract}
For any Lie algebra $L$ over a field, its universal enveloping
algebra $U(L)$ can be embedded in a division ring $\mathfrak{D}(L)$
constructed by  Lichtman. If $U(L)$ is
an Ore domain, $\mathfrak{D}(L)$ coincides with its ring of
fractions.

It is well known that the principal involution of $L$, $x\mapsto
-x$, can be extended to an involution of $U(L)$, and Cimpri\u{c} has
proved that this involution can be extended to one on
$\mathfrak{D}(L)$.

For a large class of noncommutative Lie algebras $L$ over a field of
characteristic zero, we show that $\mathfrak{D}(L)$ contains
noncommutative free algebras generated by symmetric elements with
respect to (the extension of) the principal involution. This class
contains all noncommutative Lie algebras such that $U(L)$ is an Ore domain.
\end{abstract}

\maketitle

\section*{Introduction}

A long-standing conjecture of Makar-Limanov states that a division
ring which is finitely generated (as a division ring) and infinite
dimensional over its center contains a free (associative) algebra of
rank two over the center \cite{Makaronfreesubobjects}.
Makar-Limanov's conjecture has been extensively investigated, see
e.g. \cite{GoncalvesShirvaniSurvey}. We remark that  containing a
free algebra over the center  is equivalent  to containing a free
algebra over any central subfield \cite[Lemma~1]{MakarMalcolmson91}.

In \cite{FerreiraGoncalvesSanchez1},  we raised a  question related
to Makar-Limanov's conjecture. In the presence of an involution,
$x\mapsto x^*$, a natural investigation would be whether a division
ring satisfying Makar-Limanov's conjecture contains a free algebra
of rank 2 generated by symmetric elements i.e. elements satisfying
$x^*=x$.

Let $k$ be a field. There are two important families of $k$-algebras
which are endowed with  natural involutions. The first one is the
class of group algebras $k[G]$ over a group $G$. The map
$k[G]\rightarrow k[G]$, $\sum_{x\in G}xa_x\mapsto \sum_{x\in G}
x^{-1}a_x$, where, for all $x\in G$, $a_x$ are elements of $k$ which
are zero except for a finite number of $x$, is called the
\emph{canonical involution}. The second important class of
$k$-algebras endowed with an involution is the class of the
universal enveloping algebras of Lie algebras. Let $L$ be a Lie
$k$-algebra and $U(L)$ be its universal enveloping algebra. There is
a natural involution on $L$, $L\rightarrow L$, $x\mapsto -x$. It is
well known that this involution can be extended to an involution of
$U(L)$ that fixes $k$ elementwise. We call this involution the
\emph{principal involution} on $L$.

It is not known for which groups $G$ the group algebra $k[G]$ is
embeddable in a division ring. Moreover, even if $k[G]$ is embedded
in a division ring $D$, the canonical involution may not be extended
to one on $D$. When $G$ is an orderable group, the group ring $k[G]$
is embeddable in $k((G))$, the division algebra of Malcev-Neumann
series of $G$ over $k$ \cite{Malcev}, \cite{Neumann}. Let $k(G)$ be
the division algebra of $k((G))$ generated by $k[G]$. In
\cite{SanchezfreegroupalgebraMNseries}, it was shown that if $G$ is
not abelian, any division ring containig $k[G]$ must contain a free
$k$-algebra of rank two. In \cite{FerreiraGoncalvesSanchez1}, we
proved that the canonical involution of $k[G]$ can be extended to
$k(G)$ and that $k(G)$ contains a free $k$-algebra of rank two
generated by symmetric elements with respect to the canonical
involution.

That $U(L)$, the universal enveloping algebra of a Lie $k$-algebra
$L$, can be embedded in a division algebra was proved by Cohn in
\cite{Cohnembeddingrings}. In \cite{Lichtmanvaluationmethods},
Lichtman obtained a simpler proof of this fact. For each Lie
$k$-algebra $L$, he constructed a division $k$-algebra
$\mathfrak{D}(L)$ that contains $U(L)$ and it is generated by it. In
the study of division rings, $\mathfrak{D}(L)$ plays a similar role
to $k(G)$ for the group algebra $k[G]$ of an ordered group $G$. In
\cite{Cimpricfreefieldmany}, Cimpri{\v{c} proved that the principal
involution on $U(L)$ can be extended to $\mathfrak{D}(L)$. In
\cite{Lichtmanfreeuniversalenveloping}, Lichtman also proved that if
$k$ is a field of characteristic zero and $L$ is not commutative,
any division ring that contains $U(L)$ must contain a free algebra
of rank 2. The main aim of this paper is to show that if $k$ is of
characteristic zero, then, for a large class of noncommutative Lie
$k$-algebras $L$, the division ring $\mathfrak{D}(L)$ contains a
free $k$-algebra of rank two generated by symmetric elements with
respect to the principal involution. This class contains the
noncommutative residually nilpotent Lie $k$-algebras and the
noncommutative Lie $k$ algebras $L$ for which $U(L)$ is an Ore
domain. We give explicit generators for the free algebras generated
by symmetric elements.

\medskip

The contents of Section~\ref{sec:preliminaries} are known. We sketch
the construction of $\mathfrak{D}(L)$ following
Lichtman~\cite{Lichtmanvaluationmethods}, and show how the princial
involution is extended to $\mathfrak{D}(L)$ as shown in
\cite{Cimpricfreefieldmany}. We also present some results on the
completion of skew polynomial rings that will be needed in
Section~\ref{sec:residuallynilpotent}.

Let $k$ be a field of characteristic zero. Let $H$ be the Heisenberg
Lie $k$-algebra, i.e. the Lie $k$-algebra with presentation
$$H=\langle x,y\mid [[y,x],x]=[[y,x],y]=0 \rangle.$$
The aim of Section~\ref{sec:HeisenbergLiealgebra} is to find
generators of a free $k$-algebra of rank two generated by symmetric
elements with respect to the principal involution inside
$\mathfrak{D}(H)$, the Ore division ring of fractions of $U(H)$.
These generators must be of a certain form so that the results of
Sections~\ref{sec:residuallynilpotent} and \ref{sec:U(L)isOre} can
be applied. The generators are obtained using a result of
Cauchon~\cite{Cauchoncorps}, who used the techniques from
\cite{Makar-LimanovOnsubalgebrasofthe}, and applying a technical
result proved in Subsection~\ref{subsec:technicalresult} that
symmetrizes the free algebra obtained from Cauchon's result.

In Section~\ref{sec:residuallynilpotent}, we  obtain a free algebra
of rank two generated by symmetric elements with respect to the
principal involution inside $\mathfrak{D}(L)$ where $L$ is a
(certain generalization of) residually nilpotent Lie algebra over a
field $k$ of characteristic zero. This free algebra is obtained by
pulling back the free algebra generated by symmetric elements in the
case of the Heisenberg Lie algebra. This pullback is done using a
series technique which can be seen as a translation of the one used
for the group algebra $k[G]$ of an ordered group $G$ and $k(G)$, in
\cite{SanchezfreegroupalgebraMNseries} and
\cite{FerreiraGoncalvesSanchez1}, to the language of universal
enveloping algebras and $\mathfrak{D}(L)$.

In Section~\ref{sec:U(L)isOre}, we use the technique from
\cite{Lichtmanfreeuniversalenveloping} to obtain a free algebra
inside $\mathfrak{D}(L)$, which is generated by symmetric elements
with respect to the principal involution, from the one obtained in
the residually nilpotent case.

In Section~5 further comments on the subject are made.

\bigskip

All algebras (except for Lie algebras) will be associative with $1$, and their homomorphisms will be understood to respect $1$.   We will also work with Lie algebras, but the term Lie will always be explicit so that no confusion is possible.

A \emph{valuation} on an algebra $A$ is a map $\vartheta\colon
A\rightarrow \mathbb{Z}\cup\{\infty\}$  that satisfies the following
three properties (i) $\vartheta(x)=\infty$ if and only if $x=0$,
(ii) $\vartheta(xy)=\vartheta(x)+\vartheta(y)$, and (iii)
$\vartheta(x+y)\geq \min\{\vartheta(x),\vartheta(y)\}$, for all
$x,y\in A$.

By \emph{free algebra on a set $X$}  over a field $k$, we mean the
associative $k$-algebra of noncommuting polynomials with
indeterminates from the set $X$.

Given a $k$-algebra $A$, a \emph{$k$-derivation} is a $k$-linear map
$\delta \colon A\rightarrow A$ such that
$\delta(ab)=\delta(a)b+a\delta(b)$ for all $a,b\in A$.

We shall give a few words about Ore localization. Notice that we
shall only need localization of domains at certain multiplicative
sets. For further details the reader is referred to, for example,
\cite{Lam2}.

We say that an algebra $R$ is a \emph{domain} if it is nonzero and
contains no zero divisors other than $0$.

A subset $\mathfrak{S}$ of $R$ is \emph{multiplicative} if it
contains $1$ and is closed under multiplication, i.e.
$x,y\in\mathfrak{S}$ implies that $xy\in\mathfrak{S}$. A
multiplicative subset of the domain $R$ is a \emph{left Ore} subset
if for any $a\in R$ and $s\in\mathfrak{S}$, $\mathfrak{S}a\cap
Rs\neq\emptyset$.  \emph{Right Ore} subsets are defined similarly. A
left and right Ore subset of $R$ is called an \emph{Ore subset}. We
say that a domain $R$ is an \emph{Ore domain} if the multiplicative
set $R\setminus\{0\}$ is an Ore subset of $R$.

Let $R,\ R'$ be any algebras and $\mathfrak{S}$ a subset of $R$. An
algebra homomorphism $f\colon R\rightarrow R'$ is said to be
\emph{$\mathfrak{S}$-inverting} if $f$ maps $\mathfrak{S}$ into the
units of $R'$. We say that a $\mathfrak{S}$-inverting homomorphism
$f\colon R\rightarrow R'$ is a \emph{universal
$\mathfrak{S}$-inverting homomorphism} if for any other
$\mathfrak{S}$-inverting homomorphism $g\colon R\rightarrow R''$,
there exists a unique homomorphism $\tilde{g}\colon R'\rightarrow
R''$ such that $g=\tilde{g}f$.

The results we will need are the following. Let $R$ be a domain and
$\mathfrak{S}$ a (left) Ore subset in $R$. There exists a unique
algebra $\mathfrak{S}^{-1}R$ such that $R$ embeds into
$\mathfrak{S}^{-1}R$, all the elements of $\mathfrak{S}^{-1}R$ can
be expressed in the form $s^{-1}r$ where $a\in R$,
$s\in\mathfrak{S}$, and the embedding $R\hookrightarrow
\mathfrak{S}^{-1} R$ is universal $\mathfrak{S}$-inverting.
Moreover, if $\mathfrak{S}=R\setminus\{0\}$, then $\mathfrak{S}^{-1}
R$ is a division algebra that contains $R$ and is generated by it.
In this event, we will say that $\mathfrak{S}^{-1} R$ is the
\emph{Ore division ring of fractions} of $R$.


\section{Preliminaries}\label{sec:preliminaries}

In this section, we present known results and definitions that will be useful in later sections. The lesser known results are proved in some detail for the sake of completeness.

\subsection{The universal enveloping algebra $U(L)$ and the division ring $\mathfrak{D}(L)$}

Throughout this section, $k$ will denote a field.

\medskip

Let $L$ be a Lie $k$-algebra. We will denote by $U(L)$ its
\emph{universal enveloping algebra}, see \cite{JacobsonLiealgebras} or \cite{Dixmierenvelopingalgebras} for further details.

Recall that one can obtain a Lie $k$-algebra from any associative
$k$-algebra $A$. Indeed, define a new multiplication  on $A$ by the
rule $[x,y]=xy-yx$. The $k$-vector space $A$ with this
multiplication is a Lie $k$-algebra denoted by $A^-$. The universal
enveloping algebra of a Lie $k$-algebra $L$ is an associative
$k$-algebra $U(L)$, such that $L$ is a Lie subalgebra of $U(L)^-$,
satisfying the property that given an associative $k$-algebra $A$,
any Lie algebra homomorphism $L\rightarrow A^-$ can be extended to a
homomorphism $U(L)\rightarrow A$.

In particular one obtains that any homomorphism of Lie $k$-algebras
$L\rightarrow L'$ can be uniquely extended to a homomorphism of
associative $k$-algebras $U(L)\rightarrow U(L')$. From this
universal property one also obtains the augmentation map
$\varepsilon \colon U(L)\rightarrow k$, which is the unique
extension of the Lie $k$ algebra homomorphism $L\rightarrow k^-$,
$x\mapsto 0$.

\medskip

Let $L$ be a Lie $k$-algebra. Let $\mathcal{L}$ be a $k$-linearly
independent subset of $L$. Suppose that we have defined a total
order $<$ in $\mathcal{L}$. The set of \emph{standard monomials} in
$\mathcal{L}$ is the subset of $U(L)$ consisting of the monomials of
the form $x_1x_2\dotsb x_m$ with $m\geq0$, $x_i\in\mathcal{L}$ and
$x_1\leq x_2\leq \dotsb\leq x_m$, where we understand that the
identity element in $U(L)$ is the standard monomial corresponding to
$m=0$.

Let $\mathcal{B}=\{x_i\mid i\in I\}$ be a totally ordered basis of
$L$. The Poincar\'e-Birkhoff-Witt Theorem (PBW Theorem for short)
states that the standard monomials in $\mathcal{B}$ form a $k$-basis
of $U(L)$. Thus, every element of $U(L)$ can be uniquely expressed
as a finite sum $\sum x_{i_1}x_{i_2}\cdots x_{i_r} a_{i_1i_2\cdots
i_r}$, where $x_{i_1}x_{i_2}\cdots x_{i_r}$ is a standard monomial
in $\mathcal{B}$ and $a_{i_1i_2\cdots i_r}\in k$.

\medskip

\medskip

Let $L$ be a  Lie $k$-algebra and $U(L)$ its universal enveloping
algebra. In~\cite{Lichtmanvaluationmethods}, Lichtman gives a
construction of a skew field $\mathfrak{D}(L)$ which contains $U(L)$
and it is generated by it. We recall the construction (without proofs) for later
reference. For further details, the interested reader is referred to \cite[Section~2.6]{Cohnskew}
or the original paper  \cite{Lichtmanvaluationmethods}. Most of our exposition is taken from
\cite[Section~2]{Cimpricfreefieldmany}.

The \emph{standard filtration} $$\dotsb \subseteq
F_1\subseteq F_0\subseteq F_{-1}\subseteq F_{-2}\subseteq \dotsb$$
of $U(L)$ is defined by $F_i = \{0\}$ if $i
> 0$, $F_0 = k$ and for $i = -n < 0$, $F_i$ is the vector subspace of
$U(L)$ generated by all products containing  $\leq n$ elements from
$L$. The mapping $\vartheta \colon U(L)\rightarrow
\mathbb{Z}\cup\{\infty\}$ defined by $\vartheta(f) = \sup\{i \mid f
\in F_i\}$ is a valuation called the \emph{standard valuation}. Let
$U(L)[t, t^{-1}]$ be the ring of Laurent polynomials in a central
variable $t$. Extend $\vartheta$ to a valuation on $U(L)[t, t^{-1}]$
by $\vartheta( \sum_i t^if_i ) = \min\{\vartheta(f_i) + i\mid
i\in\integers\}$. Define $T=\{h\in U(L)[t,t^{-1}] \mid
\vartheta(h)\geq 0\}\subseteq U(L)[t]$. Note that $U(L)[t, t^{-1}]$
is canonically isomorphic to $\mathcal{C}^{-1}T$, the Ore
localization of $T$ at the multiplicative set
$\mathcal{C}=\{1,t,t^2,\dotsc\}$. For every $n = 1, 2,\dots$, write
$T_n = T/t^nT$ and $\mathcal{U}_n$ for the projection of
$\mathcal{U} = T \setminus tT$ onto $T_n$. It turns out that
$\mathcal{U}_n$ is a regular Ore set in $T_n$ for every $n$ and we
denote by $S_n =\mathcal{U}_n^{-1} T_n$ its Ore localization. There
are natural surjective $k$-algebra homomorphisms
$$\dotsb\rightarrow S_{n+1}\rightarrow S_n\rightarrow \dotsb\rightarrow S_1$$
Write $S$ for the inverse limit of $S_n$. The localization $D
=\mathcal{C}^{-1} S$ is a division ring. There exists a natural
embedding of $T$ into $S$ which extends to a natural embedding of
$U(L)[t, t^{-1}] =\mathcal{C}^{-1}T$ into $D = \mathcal{C}^{-1} S$.
Write $\mathfrak{D}(L)$ for the minimal division ring of $D$
containing $U(L)$. We remark that if $U(L)$ is an Ore domain, then $\mathfrak{D}(L)$
is the Ore division ring of fractions of $U(L)$. Indeed, the universal property
of Ore localizations implies that there exists a $k$-algebra homomorphism from the Ore
division ring of fractions of $U(L)$, but $\mathfrak{D}(L)$ is a division ring that contains $U(L)$
and is generated by it. Thus they are isomorphic.

\medskip

Now we turn our attention to involutions.

Let $R$ be an associative $k$-algebra. A $k$-linear map $*\colon
R\rightarrow R$ is a $k$-\emph{involution} if for all $x,y\in R$,
$(xy)^*=y^*x^*$, $x^{**}= x$.

The analogous concept for Lie algebras is as follows. Let $L$ be a
Lie $k$-algebra. A $k$-linear map  $*\colon L\rightarrow L$ is a
$k$-\emph{involution} if for all $x,y\in L$, $[x,y]^*=[y^*, x^*]$,
$x^{**}= x$. The main example of a $k$-involution in a Lie
$k$-algebra is what we call the \emph{principal involution}. It is
defined by $x\mapsto -x$ for all $x\in L$.

The following proposition is a slight variation of
\cite[Proposition~5]{Cimpricfreefieldmany} whose proof we follow.

\begin{prop}
Every $k$-involution on a Lie $k$-algebra $L$ has a canonical
extension to $\mathfrak{D}(L)$. In particular, the principal
involution of $L$ can be canonically extended to a $k$-involution of $\mathfrak{D}(L)$.
\end{prop}

\begin{proof} Every $k$-involution extends uniquely from $L$ to its enveloping
algebra $U(L)$ \cite[Section~2.2.17]{Dixmierenvelopingalgebras}. Setting $t^*= t$, we get an involution on $U(L)[t,
t^{-1}]$ which induces an involution on $T$. Since
$(t^nT)^*\subseteq t^nT$ for every $n$, we have an induced
involution on $T/t^nT$ which will also be denoted by $*$. Note that
the natural epimorphisms $\phi_n:T_{n+1}\rightarrow T_n$ are
$*$-homomorphisms.

One can prove that $\mathcal{U}^*\subseteq \mathcal{U}$ and
$\mathcal{U}_n^*\subseteq \mathcal{U}_n$. Then the involution on
$T_n$ extends uniquely to an involution of $S_n =\mathcal{U}_n^{-1}
T_n$. It is easy to verify that the natural epimorphisms
$\phi'_n\colon S_{n+1}\rightarrow S_n$ are $*$-homomorphisms. It
follows that the termwise involution on the inverse system of $S_n$
and $\phi_n'$ induces an involution on its inverse limit $S$. Since
$t^*= t$, this involution extends uniquely to an involution on $D =
\mathcal{C}^{-1}S$.  Since $\mathfrak{D}(L)\cap \mathfrak{D}(L)^*$
is a division subring of $D$ containing $U(L) = U(L)^*$ and since
$\mathfrak{D}(L)$ is the smallest subfield of $D$ containing $U(L)$,
it follows that $\mathfrak{D}(L) =\mathfrak{D}(L)^*$.
\end{proof}

\medskip

Some other important properties of $\mathfrak{D}(L)$ that we will
need in Section~\ref{sec:residuallynilpotent} are contained in the
following result \cite[Proposition~2.5]{Lichtmanuniversalfields}.

\begin{prop}\label{prop:canonicalfieldoffractions}
Let $L$ be a Lie $k$-algebra. Suppose that $N$ is a Lie subalgebra
of $L$. The following properties are satisfied:
\begin{enumerate}[\rm (1)]
\item The following diagram is
commutative $$\xymatrix{ U(N)\ar@{^{(}->}[r] \ar@{^{(}->}[d] &
\mathfrak{D}(N)\ar@{^{(}->}[d]\\
U(L) \ar@{^{(}->}[r] & \mathfrak{D}(L)}$$
\item If $\mathcal{B}_N$ is a basis of $N$ and $\mathcal{C}$
is a set of elements of $L\setminus N$ such that $\mathcal{B}_N\cup
\mathcal{C}$  is a basis of $L$, then the standard monomials in
$\mathcal{C}$
 are linearly independent over
$\mathfrak{D}(N)$. \qed
\end{enumerate}
\end{prop}

\subsection{Formal differential  and pseudo-differential operator algebras}
\label{sec:differentialoperators}

Let $R$ be a $k$-algebra,  $\delta$ a $k$-derivation on $R$ and
$\sigma$ a $k$-algebra automorphism of $R$. The \emph{formal
differential operator algebra} over $R$ (respectively, the
\emph{skew polynomial algebra} over $R$), denoted by $R[x;\delta]$
(respectively, $R[x;\sigma]$), is defined as  the $k$-algebra $S$
such that
\begin{enumerate}[(a)]
\item $S$ is a ring containing $R$ as a $k$-subalgebra.
\item $x$ is an element of $S$.
\item $S$ is a free right $R$-module with basis $\{1,x,x^2,\dotsc\}$.
\item $ax=xa+\delta(a)$ for all $a\in R$ (respectively, $ax=x\sigma(a)$ for all $a\in R$.)
\end{enumerate}

The $k$-algebra $R[x;\delta]$ has the following universal property.
Given a $k$-algebra $T$, a $k$-algebra homomorphism $\phi\colon
R\rightarrow T$ and an element $y\in T$ such that
\begin{equation}\label{eq:univpropertypolynomial}
\phi(a)y=y\phi(a)+\phi(\delta(r)) \ \ \textrm{for all } a\in R,
\end{equation}
there exists a unique $k$-algebra homomorphism $\psi\colon
R[x;\delta]\rightarrow T$ such that $\psi_{|R}=\phi$ and
$\psi(x)=y$. There is an analogous universal property for
$R[x;\sigma]$, but we will not need it.

Given $R[x;\delta]$, one can construct the \emph{formal
pseudo-differential operator ring}, denoted $R((t_x;\delta))$,
consisting of the formal Laurent series $\sum_{i=n}^\infty
t_x^ia_i$, with $n\in\mathbb{Z}$ and coefficients $a_i\in R$,
satisfying $at_x^{-1}=t_x^{-1}a+\delta(a)$ for all $a\in R$.
Therefore
\begin{equation}
at_x=t_xa-t_x\delta(a)t_x = \sum_{i=1}^\infty
t_x^i(-1)^{i-1}\delta^{i-1}(a),
\end{equation}
for any $a\in R$.

The subset $R[[t_x;\delta]]$ of $R((t_x;\delta))$ consisting of the
Laurent series of the form $\sum_{i=0}^\infty t_x^ia_i$ is a
$k$-subalgebra of $R((t_x;\delta))$.

The ring $R[[t_x;\delta]]$ can be regarded in another way. Define
$R\langle t_x;\delta\rangle$ as the $k$-algebra $R\langle t_x \mid
at_x=t_xa-t_x\delta(a)t_x,\ a\in R \rangle$. In other words, the
$k$-algebra  $R\langle t_x;\delta\rangle$ is isomorphic to the
coproduct $R\coprod _k k[t_x]$ modulo the two-sided ideal generated
by $\{at_x=t_xa-t_x\delta(a)t_x,\ a\in R \}$. By definition, the
$k$-algebra $R\langle t_x;\delta\rangle$ has the following universal
property. Suppose that we have a $k$-algebra $T$, a $k$-algebra
homomorphism $\phi\colon R\rightarrow T$ and an element $t_y\in T$
such that
\begin{equation}\label{eq:univpropertypolynomialt}
\phi(a)t_y=t_y\phi(a)-t_y\phi(\delta(a))t_y \ \ \textrm{ for all }
a\in R.
\end{equation}
Then there is a unique $k$-algebra homomorphism $\psi\colon R\langle
t_x;\delta\rangle\rightarrow T$ such that $\psi_{|R}=\phi$ and
$\psi(t_x)=t_y$.

Let $\varepsilon_1\colon k[t_x]\to R\langle t_x;\delta\rangle$ and
$\varepsilon _2 \colon R\to R\langle t_x;\delta\rangle$ be the
natural homomorphisms of $k$-algebras. By the universal property,
there is a $k$-algebra homomorphism $\varphi\colon R\langle
t_x;\delta\rangle\to R[[t_x;\delta]]$ such that, for any $n\ge 1$,
$\varphi (\varepsilon _1(t_x^n))=t_x^n$ and, for any $a\in R$,
$\varphi (\varepsilon _2(a))=a$. Therefore $\varepsilon _1$ and
$\varepsilon _2$ are injective homomorphisms. To simplify the
notation, we just identify $R$ and $k[t_x]$ with their images in
$R\langle t_x;\delta\rangle$ without making any reference to the
embeddings $\varepsilon _1$ and $\varepsilon _2$.

Given a positive integer $n$, any element  $f\in R\langle
t_x;\delta\rangle$ can be expressed as
$$f=a_0+t_xa_1+\dotsb+t_x^{n-1}a_{n-1}+t^nf_n,$$ where $f_n\in
R\langle t_x;\delta\rangle$, and unique $a_0,\dotsc,a_{n-1}\in R$.
Hence, the $k$-algebra $R[[t_x\delta]]$ is the completion of
$R\langle t_x;\delta\rangle$ with respect to the topology induced by
the chain of ideals $\{t_x^nR\langle t_x;\delta\rangle \}_{n\geq 0}$
(see \cite[Lemma~7.3]{Herberasanchezinfiniteinversion} for more
details). Moreover the set $\mathcal{S}=\{1,t_x,t_x^2,\dots\}$ is a
left denominator set of $R[[t_x;\delta]]$ such that the Ore
localization $\mathcal{S}^{-1}R[[t_x;\delta]]$ is the $k$-algebra
$R((t_x;\delta))$, see for example \cite[Theorem~2.3.1]{Cohnskew}.

Note that there is a natural embedding $R[x;\delta]\hookrightarrow
R((t_x;\delta_x))$ sending $x$ to $t_x^{-1}$.

The following lemma will be useful.

\begin{lem}\label{lem:extensionmorphismtocompletion}
Let $R$ and $S$ be two $k$-algebras. Suppose that $\delta_w\colon
R\rightarrow R$ and $\delta_z\colon S\rightarrow S$ are
$k$-derivations of $R$ and $S$, respectively. Let $\phi\colon
R\rightarrow S$ be a $k$-algebra homomorphism such that
\begin{equation}\label{eq:conditionforextension}
\phi(\delta_w(a))=\delta_z(\phi(a)) \ \ \textrm{for all } a\in R.
\end{equation}
Then
\begin{enumerate}[\rm (1)]
\item $\phi$ can be extended to a unique $k$-algebra homomorphism
$$R[w;\delta_w]\rightarrow S[z;\delta_z]$$ sending $w$ to $z$.
\item $\phi$ can be extended to a unique $k$-algebra homomorphism
$$\psi:R\langle t_w;\delta_w\rangle \rightarrow S\langle t_z;\delta_z\rangle,$$
such that  $\psi(t_w)=t_z$.
\item The $k$-algebra homomorphism $\psi:R\langle t_w;\delta_w\rangle \rightarrow S\langle
t_z;\delta_z\rangle$ induces a $k$-algebra homomorphism
$$\Psi\colon R((t_w;\delta_w))\rightarrow S((t_z;\delta_z)),\ \
\sum_i t_w^ia_i\mapsto \sum_i t_z^i\phi(a_i).$$
\end{enumerate}
\end{lem}

\begin{proof}
(1)  follows from the universal property of $R[w;\delta_w]$.

(2) follows from the universal property of $R\langle
t_w;\delta_w\rangle$.

(3) The homomorphism $\psi$ is such that $\psi(a)=\phi(a)$ for all
$a\in R$ and $\psi(t_w)=t_z$. Hence it induces the commutative
diagram of $k$-algebra homomorphisms
$$\xymatrix{ \frac{R\langle t_w;\delta_w\rangle}{t_w^{n+1}R\langle t_w;\delta_w\rangle } \ar[r]\ar[d]  &
\frac{S\langle t_z;\delta_z\rangle}{t_z^{n+1}S\langle t_z;\delta_z\rangle  }\ar[d]\\
   \frac{R\langle t_w;\delta_w\rangle}{t_w^{n}R\langle t_w;\delta_w\rangle } \ar[r]       &
 \frac{S\langle t_z;\delta_z\rangle}{t_z^n S\langle t_z;\delta_z\rangle } }$$
where the vertical arrows are the canonical projections and the
horizontal arrows are induced by $\psi$ in the natural way.

The universal property of the completion yields the $k$-algebra
homomorphism $\Phi\colon R[[t_w;\delta_w]]\rightarrow
S[[t_z;\delta_z]]$ with $\Phi(\sum_{i=0}^\infty
t_w^ia_i)=\sum_{i=0}^\infty t_z^i\phi(a_i)$. Since
$R((t_w;\delta_w))$ is the Ore localization of $R[[t_w;\delta_w]]$
at $\mathcal{S}=\{1, t_w,t_w^2,\dots\}$, we get the desired
extension.
\end{proof}

A slight variation of the next result can be found in
\cite{Herberasanchezinfiniteinversion}.

\begin{lem}\label{lem:extensiontocompletion}
Let $R$ be a $k$-algebra, let $\delta_w\colon R\rightarrow R$ be a
$k$-derivation on $R$ and consider the formal differential operator
algebra $R[w;\delta_w]$. Let $S$ be a $k$-algebra. Consider the
inner $k$-derivation, defined by an element $s\in S$,
$\delta_s\colon S\rightarrow S$ given by $\delta_s(b)=bs-sb$ for all
$b\in S$.

Suppose that the following properties hold:
\begin{enumerate}[\rm (a)]
\item $R[w;\delta_w]$ is a subring of $S$,
\item $\delta_s(R)\subseteq R$, and
\item $\delta_s(w)\in R$.
\end{enumerate}
Then the following assertions hold true.
\begin{enumerate}[\rm (1)]
\item $\delta_s$ induces  a $k$-derivation
$\delta_s\colon R\langle t_w;\delta_w\rangle\rightarrow R\langle
t_w;\delta_w\rangle$ such that $\delta_s(t_w)=-t_w\delta_s(w)t_w$.
\item $\delta_s$ induces a $k$-derivation
$\delta_s\colon R((t_w;\delta_w))\rightarrow R((t_w;\delta_w))$ such
that $\delta_s\left(\sum_{i\geq 0} t_w^ia_i\right)=\sum_{i\geq
0}\delta_s(t_w^ia_i)$, where
$\delta_s(t_w)=-t_w\delta_s(w)t_w=\sum_{i\geq 1}
t_w^{i+1}(-1)^{i}\delta_w^{i-1}(\delta_s(w))$, and
$\delta_s(t_w^{-1})=\delta_s(w)$.
\end{enumerate}
\end{lem}

\begin{proof}
(1) We must show that there exists a homomorphism of $k$-algebras
$\Phi\colon R\langle t_w;\delta_w\rangle\rightarrow
\mathbb{T}_2(R\langle t_w;\delta_w\rangle)$, $f\mapsto \bigl(
\begin{smallmatrix} f & \delta_s(f) \\ 0 & f
\end{smallmatrix}\bigr)$, where $\mathbb{T}_2(R\langle t_w;\delta_w\rangle)$ is the ring of
$2\times 2$ upper triangular matrices over $R\langle
t_w;\delta_w\rangle$.

By the universal property \eqref{eq:univpropertypolynomialt}, this
is equivalent to proving that, for any $a\in R$, the following
matrix equality holds
$$\left(\begin{array} {cc}a & \delta_s(a) \\ 0 & a  \end{array} \right)
\left(\begin{array}{cc}t_w & \delta_s(t_w) \\ 0 & t_w \end{array}
\right)
$$
$$=
\left(\begin{array}{cc} t_w & \delta_s(t_w)
\\ 0 & t_w   \end{array} \right)\left(\begin{array}{cc} a & \delta_s(a) \\ 0 & a
\end{array} \right)  -
 \left(\begin{array}{cc} t_w & \delta_s(t_w) \\ 0 & t_w \end{array} \right)\left(\begin{array}{cc} \delta_w(a) & \delta_s(\delta_w(a))  \\ 0 & \delta_w(a)  \end{array}
\right)\left(\begin{array}{cc} t_w & \delta_s(t_w) \\ 0 & t_w
\end{array} \right), $$
\medskip

\noindent which  yields
$$\left(\begin{array}{cc} at_w &  \delta_s(a)t_w + a\delta_s(t_w) \\ 0 & at_w \end{array} \right)=$$ $$=
\left(\begin{array}{cc} t_wa &  \delta_s(t_w)a+t_w\delta_s(a) \\ 0 &
t_wa
\end{array} \right) - \left(\begin{array}{cc}  t_w\delta_w(a)t_w
&   t_w\delta_w(a)\delta_s(t_w)+ t_w\delta_s(\delta_w(a))t_w +
\delta_s(t_w)\delta_w(a)t_w  \\  0 & t_w\delta_w(a)t_w\end{array}
\right).$$

Since $at_w=t_wa-t_w\delta_w(a)t_w$, the matrix equality is
equivalent to the equality
$$\delta_s(a)t_w + a\delta_s(t_w) \stackrel{(*)}{=}\delta_s(t_w)a+t_w\delta_s(a)
-t_w\delta_w(a)\delta_s(t_w)- t_w\delta_s(\delta_w(a))t_w -
\delta_s(t_w)\delta_w(a)t_w.$$

After substituting  $-t_w\delta_s(w)t_w$ for $\delta_s(t_w)$, the
right hand side of the equality $(*)$ is equal to
\begin{eqnarray*}
& &
-t_w\delta_s(w)t_wa+t_w\delta_s(a)+t_w\delta_w(a)t_w\delta_s(w)t_w-t_w\delta_s(\delta_w(a))t_w
+ t_w\delta_s(w)t_w\delta_w(a)t_w\\
& = & -t_w\delta_s(w)(at_w+t_w\delta_w(a)t_w)
+t_w\delta_s(a)+t_w\delta_w(a)t_w\delta_s(w)t_w-t_w\delta_s(\delta_w(a))t_w
+ t_w\delta_s(w)t_w\delta_w(a)t_w \\
& = & -t_w\delta_s(w)at_w
+t_w\delta_s(a)+t_w\delta_w(a)t_w\delta_s(w)t_w-t_w\delta_s(\delta_w(a))t_w
\end{eqnarray*}
Now, the left hand side of $(*)$ is
\begin{eqnarray*}
\delta_s(a)t_w+a\delta_s(t_w)  & = & t_w\delta_s(a)-t_w\delta_w(\delta_s(a))t_w-at_w\delta_s(w)t_w \\
& = & t_w\delta_s(a)-t_w\delta_w(\delta_s(a))t_w-t_wa\delta_s(w)t_w+t_w\delta_w(a)t_w\delta_s(w)t_w \\
\end{eqnarray*}
After eliminating equal terms on both sides of $(*)$, we see that it
holds if and only if
$$-t_w\delta_s(w)at_w -t_w\delta_s(\delta_w(a))t_w= -t_w\delta_w(\delta_s(a))t_w-t_wa\delta_s(w)t_w.$$
Equivalently, \begin{equation}\label{eq:commutators}
-t_w([w,s],a]]+[[a,w],s]+[[s,a],w] )t_w=0.\end{equation} By the
Jacobi identity,  equality \eqref{eq:commutators} holds. Therefore
$\Phi \colon S \to \mathbb{T}_2(S)$  must be a $k$-algebra
homomorphism. This shows the existence of  the claimed  derivation
$\delta _s\colon R\langle t_w;\delta_w\rangle\to R\langle
t_w;\delta_w\rangle$.

(2) By (1), the map $\Phi\colon R\langle
t_w;\delta_w\rangle\rightarrow \mathbb{T}_2(R\langle
t_w;\delta_w\rangle)$, defined by $f\mapsto
\left(\begin{smallmatrix} f & \delta_s(f) \\ 0 &
f\end{smallmatrix}\right)$ is a $k$-algebra homomorphism. Note that
$\mathbb{T}_2(R\langle t_w;\delta_w\rangle)$ is canonically
isomorphic to $\mathbb{T}_2(R)\langle T_w;\Delta_w\rangle$ where
$T_w=\left(\begin{smallmatrix} t_w & 0 \\ 0 &
t_w\end{smallmatrix}\right)$ and $\Delta_w\colon
\mathbb{T}_2(R)\rightarrow \mathbb{T}_2(R)$ is defined by
$\left(\begin{smallmatrix} a & b \\ 0 &
c\end{smallmatrix}\right)\mapsto \left(\begin{smallmatrix} \delta_w(a) & \delta_w(b) \\
0 & \delta_w(c)\end{smallmatrix}\right)$. Note also that
$\mathbb{T}_2(R[[ t_w;\delta_w]])$ is canonically isomorphic to
$\mathbb{T}_2(R)[[ T_w;\Delta_w]]$.

One can prove that $\Phi(t_w^nR\langle t_w;\delta_w\rangle
)\subseteq T_w^n\mathbb{T}_2(R)\langle T_w;\Delta_w\rangle $ for
each positive integer $n$. These induce a $k$-algebra homomorphism
between the completions $R[[t_w;\delta_w]]\rightarrow
\mathbb{T}_2(R[[ t_w;\delta_w]])$ defined by $\delta_s\left(\sum_i
t_w^ia_i\right)=\sum_i\delta_s(t_w^ia_i)$. This morphism extends to
$R((t_w;\delta_w))$, as desired.
\end{proof}


\section{The case of the Heisenberg Lie algebra}\label{sec:HeisenbergLiealgebra}

The \emph{Heisenberg Lie $k$-algebra} is the Lie $k$-algebra with
presentation
\begin{equation}\label{eq:Heisenberg}
H=\langle x,y\mid [[y,x],x]=[[y,x],y]=0 \rangle.
\end{equation}
The Heisenberg Lie $k$-algebra  can also be characterized as the
unique Lie $k$-algebra of dimension three such that $[H,H]$ has
dimension one and $[H,H]$ is contained in the center of $H$, see
\cite[Section~4.III]{JacobsonLiealgebras}.

Our aim in this section is to prove that $\mathfrak{D}(H)$ contains free algebras generated by symmetric elements with respect to the principal involution.

\subsection{A technical result}\label{subsec:technicalresult}

In this subsection, we prove a result that will be useful
to obtain free algebras generated by symmetric elements.

\medskip

An \emph{ordered group} $(G,<)$ is a group $G$ together with a total order relation $<$
such that the product in $G$ is compatible with $<$. More precisely,
the inequality $x<y$ implies that $zx<zy$ and $xz<yz$ for all $x,y,z\in G$.

\begin{lem}\label{lem:freeinsidegroupring}
Let $(G,<)$ be an ordered group. Let $x,y\in G$ be different
positive elements in $G$ (i.e. $1<x,y$) such that the monoid
generated by them is the free monoid on the set $\{x,y\}$. Let $k$
be a field and consider the group algebra $k[G]$. Then $\{x+x^{-1},
y+y^{-1}\}$ generate a free $k$-algebra on the set $\{x+x^{-1},
y+y^{-1}\}$ inside $k[G]$.
\end{lem}

\begin{proof}
Define $X=x+x^{-1}$ and $Y=y+y^{-1}$.

 Let $M_1, M_2,\dotsc, M_r$ be
different (monomial) words on two letters and let
$a_1,a_2,\dotsc,a_r\in k\setminus\{0\}$. We have to show that
$$M_1(X,Y)a_1+\dotsb+M_r(X,Y)a_r\neq 0.$$

First note that $M_{i}(x,y)\neq M_{j}(x,y)$ if $i\neq j$ because the
monoid generated by $\{x,y\}$ is free. Thus we may suppose that
$M_1(x,y)>M_2(x,y)>\dotsb >M_r(x,y).$

Now notice that
\begin{equation}\label{eq:powers}
X^n=(x+x^{-1})^n=x^n+\sum_{-n\leq i<n}x^i\alpha_i, \qquad
Y^m=(y+y^{-1})^m=y^m+\sum_{-m\leq j<m}y^j\beta_j,
\end{equation}
for some $\alpha_i,\beta_j\in k$.

Recall that in the ordered group $G$, if $c<d$ and $e<f$, then
$ce<df$. This and \eqref{eq:powers} imply that, for fixed
$j\in\{1,\dotsc,r\}$, $M_j(X,Y)=M_j(x,y)+\sum_{l} N_lb_l$ where
$N_l\in G$, $b_l\in k$ and $M_j(x,y)>N_l$ for all $l$.

Hence $M_1(X,Y)a_1+\dotsb+M_r(X,Y)a_r=M_1(x,y)a_1+\sum_t P_tc_t$
where $P_t\in G$, $c_t\in k$ and $M_1(x,y)>P_t$ for all $t$.
Therefore, $M_1(X,Y)a_1+\dotsb+M_r(X,Y)a_r\neq 0$, as desired.
\end{proof}

\begin{coro}\label{coro:freeinsidegroupring}
Let $G$ be the free group on the set of two elements $\{x,y\}$. Let
$k$ be a field and consider the group algebra $k[G]$. Then the
$k$-algebra generated by $x+x^{-1}$ and $y+y^{-1}$ inside $k[G]$ is
free on $\{x+x^{-1}, y+y^{-1}\}$.
\end{coro}

\begin{proof}
By Lemma~\ref{lem:freeinsidegroupring}, it is enough to show that
there exists an ordered group structure $(G,<)$ on $G$ sucht
that $1<x,y$. It is
well known that this can be done. We sketch a proof of this fact.

Let $G_1=G$ and $G_{i+1}=[G,G_i]$  for $i\geq 1$.

We choose a special ordering for the group $G_1/G_2$ which is free
abelian with basis $\{xG_2, yG_2\}$. We can make
$x^{n_1}y^{m_1}G_2>x^{n_2}y^{m_2}G_2$ if either $n_1>n_2$, or
$n_1=n_2$ and $m_1>m_2$.

It is known that $G_i/G_{i+1}$ is a free abelian group of finite
rank. Thus we can order these groups, for example lexicographically.
These orderings induce an ordering in $G$ such that $x>1$ and $y>1$,
see for example \cite[p. 97]{Lam1}.
\end{proof}

\smallskip

\subsection{Cauchon's result and consequences}\label{sec:Cauchon}

In this subsection, we present the main result in \cite{Cauchoncorps}
and a consequence of it that will be useful for our purposes.

\medskip

Let $k$ be a field. Let $K=k(t)$ be the field of fractions of the
polynomial ring $k[t]$ in the variable $t$. Let $\sigma$ be a
$k$-automorphism of $K$. We will consider the skew polynomial ring
$K[p;\sigma]$. The elements of $K[p;\sigma]$ are ``right
polynomials'' of the form $\sum_{i=0}^n p^ia_i$, where the
coeficients $a_i$ are in $K$. The multiplication is determined by
$$ap=p\sigma(a) \quad \textrm{for all } a\in K.$$
It is known that $K[p;\sigma]$ is a noetherian domain and therefore
it has an Ore division ring of fractions $D=K(p;\sigma)$.

Since $\sigma$ is an automorphism of $K$,
$\sigma(t)=\frac{at+b}{ct+d}$ where $M=\left(\begin{smallmatrix} a & b \\
c & d
\end{smallmatrix}\right)\in GL_2(k)$ defines a homography $h$ of the
projective line $\Delta=\mathbb{P}_1(k)=k\cup\{\infty\}$, $h\colon
\Delta \rightarrow \Delta,$ $z\mapsto h(z)=\frac{az+b}{cz+d}$.

We denote by $\mathcal{H}=\{h^n\mid n\in\integers \}$ the subgroup
of the projective linear group $PGL_2(k)$ generated by $h$. The
group $\mathcal{H}$ acts on $\Delta$. If $z\in\Delta$, we denote by
$\mathcal{H}\cdot z=\{h^n(z)\mid n\in\integers\}$ the orbit of $z$
under the action of $\mathcal{H}$.

\begin{Cauchon}
Let $\alpha$ and $\beta$ be two elements of $k$ such that the orbits
$\mathcal{H}\cdot \alpha$ and $\mathcal{H}\cdot \beta$ are infinite
and different. Let $s$ and $u$ be the two elements of $D$ defined by
$$s=(t-\alpha)(t-\beta)^{-1}, \quad u=(1-p)(1+p)^{-1}.$$
If the characteristic of $k$ is different from $2$, then the
$k$-subalgebra $\Omega$ of $D$ generated by $\xi=s$,
$\eta=usu^{-1}$, $\xi^{-1}$ and $\eta^{-1}$ is the free group
$k$-algebra on the set $\{\xi,\eta\}$. \qed
\end{Cauchon}

We will need the following consequence of Cauchon's Theorem.

\begin{prop}\label{prop:freealgebrainWeyl}
Let $k$ be a field of  characteristic zero and $K=k(t)$ the field of
fractions of the polynomial ring $k[t]$. Let $\sigma\colon
K\rightarrow K$ be the automorphism of $k$-algebras determined by
$\sigma(t)=t-1$. Consider the skew polynomial ring $K[p;\sigma]$ and
its Ore division ring of fractions $K(p;\sigma)$. Set
$$s=\Big(t-\frac{5}{6}\Big)\Big(t-\frac{1}{6}\Big)^{-1}, \quad u=(1-p^2)(1+p^2)^{-1}.$$
Then the $k$-subalgebra of $K(p;\sigma)$ generated by $\{s+s^{-1},\
u(s+s^{-1})u^{-1}\}$ is the free $k$-algebra on the set
$\{s+s^{-1},\ u(s+s^{-1})u^{-1}\}$.
\end{prop}

\begin{proof}
We will apply Cauchon's Theorem to the skew polynomial ring
$K[p^2;\sigma^2]$, where $\sigma^2\colon K\rightarrow K$ is given by
$\sigma^2(t)=t-2$.

Let $\alpha=\frac{5}{6}\in k$ and $\beta=\frac{1}{6}$. Let
$\mathcal{H}$ be defined as above. Consider the orbits
$\mathcal{H}\cdot\alpha=\{\frac{5}{6}-2n\mid n\in\integers\}$,
$\mathcal{H}\cdot\beta=\{\frac{1}{6}-2n\mid n\in\integers\}$ which
are infinite and different.

Then, by Cauchon's Theorem, $s=(t-\alpha)(t-\beta)^{-1}$ and
$u=(1-p^2)(1+p^2)^{-1}$ are such that the $k$-algebra generated by
$\xi=s, \eta=usu^{-1}$, $\xi^{-1}$ and $\eta^{-1}$ is the free group
$k$-algebra on the free generators $\{\xi,\eta\}$.

By Corollary~\ref{coro:freeinsidegroupring}, the $k$-algebra
generated by  $\{s+s^{-1},\ u(s+s^{-1})u^{-1}\}$ is the
free $k$-algebra on the set $\{s+s^{-1},\ u(s+s^{-1})u^{-1}\}$.
\end{proof}

\smallskip


\subsection{The Heisenberg Lie algebra case}

We will work with the skew polynomial ring considered in
Section~\ref{sec:Cauchon}. More precisely, let $K=k(t)$ be the field
of fractions of the polynomial ring $k[t]$. Let $\sigma \colon
K\rightarrow K$ be the $k$-automorphism determined by
$\sigma(t)=t-1$. Consider the skew polynomial ring $K[p;\sigma]$ and
its Ore division ring of fractions $K(p;\sigma)$.

Let now $A_1=k\langle X,Y\mid YX-XY=1\rangle$ be the first Weyl
algebra, and let $D_1$ be its Ore division ring of fractions. The
following is well known.

\begin{lem}
There exists a $k$-algebra isomorphism $\psi\colon D_1\rightarrow
K(p;\sigma)$ such that $\psi(Y)=p$ and
$\psi(X)=p^{-1}t$.
\end{lem}

\begin{proof}
Since $p(p^{-1}t)-(p^{-1}t)p=t-p^{-1}p(t-1)=1$, there exists a
morphism of $k$-algebras $\psi\colon A_1\rightarrow K(p;\sigma)$
such that $\psi(Y)=p$ and $\psi(X)=p^{-1}t$. This morphism is
injective because $A_1$ is a simple ring. Moreover, by the universal
property of Ore localizations, it can be extended to a morphism of
$k$-algebras $\psi\colon D_1\rightarrow K(p;\sigma)$ because $\psi$
was injective and $K(p;\sigma)$ is a division ring. Clearly $\psi$
is injective for $D_1$ is a division ring. Now $\psi$ is surjective
because $\psi(Y)=p$ and $\psi(YX)=\psi(Y)\psi(X)=pp^{-1}t=t$.
\end{proof}

Recall the Heisenberg Lie algebra $H$ given by the presentation
\eqref{eq:Heisenberg} and let $z=[y,x]$.  The following is the left
version of \cite[Lemma~7]{Lichtmanfreeuniversalenveloping}.

\begin{lem}
There exists a morphism of $k$-algebras $\Upsilon\colon
U(H)\rightarrow D_1$ defined by $\Upsilon(x)=X$, $\Upsilon(y)=Y$ and
$\Upsilon(z)=1$. The morphism $\Upsilon$ has the following
properties:
\begin{enumerate}[\rm(1)]
\item The kernel of $\Upsilon$ is $I=U(H)(z-1)$, the ideal of $U(H)$ generated by
the central element $z-1$.

\item The set $\mathfrak{S}=U(H)\setminus I$ is an Ore subset of $U(H)$ and the Ore localization of $U(H)$ at $\mathfrak{S}$, $\mathfrak{S}^{-1}U(H)$,
is a local ring with maximal ideal $\mathfrak{S}^{-1}I$.
\end{enumerate}

\noindent Therefore $\Upsilon$ can be extended to a surjective
morphism of $k$-algebras $\Upsilon\colon
\mathfrak{S}^{-1}U(H)\rightarrow D_1$. \qed
\end{lem}

\medskip

Now we are ready to prove the main result of this section. During the proof, we will follow the notation of the foregoing results.

\begin{theo}\label{theo:freesymmetricHeisenberg}
Let $k$ be a field of  characteristic zero. Let $H$ be the
Heisenberg Lie $k$-algebra, i.e.
$$H=\langle x,y\mid [[y,x],x]=[[y,x],y]=0 \rangle.$$
Let $U(H)$ be the universal enveloping algebra of $H$ and
$\mathfrak{D}(H)$ be the Ore division ring of fractions of $U(H)$.
Set $z=[y,x]$, $V=\frac{1}{2}z(xy+yx)z$, and consider the following
elments of $\mathfrak{D}(H)$:
$$S=(V-\frac{1}{3}z^3)(V+\frac{1}{3}z^3)^{-1} +
(V-\frac{1}{3}z^3)^{-1}(V+\frac{1}{3}z^3),$$
$$T=(z+y^2)^{-1}(z-y^2)S(z+y^2)(z-y^2)^{-1}.$$ The following hold
true.
\begin{enumerate}[\rm(1)]
\item The $k$-subalgebra of $\mathfrak{D}(H)$ generated by $S$
and $T$ is the free $k$-algebra on the set $\{S,T\}$.
\item The elements $S$ and $T$ are symmetric with respect to the
principal involution on $\mathfrak{D}(H)$.
\end{enumerate}
\end{theo}

\begin{proof}
Consider the morphism of rings $\Phi=\psi\Upsilon\colon
\mathfrak{S}^{-1}U(H)\rightarrow K(p;\sigma)$. Recall that in
$K(p;\sigma)$, we have and $tp=p(t-1)$.

First notice that $\Phi(y)=p$, $\Phi(x)=p^{-1}t$ and $\Phi(z)=1$.
Thus $\Phi(V)=\Phi(\frac{1}{2}z(xy+yx)z)=t-\frac{1}{2}$,
$\Phi(V-\frac{1}{3}z^3)=t-\frac{5}{6}$,
$\Phi(V+\frac{1}{3}z^3)=t-\frac{1}{6}$, $\Phi(z+y^2)=1+p^2$,
$\Phi(z-y^2)=1-p^2$. Hence $V-\frac{1}{3}z^3$, $V+\frac{1}{3}z^3$,
 $z+y^2$, $z-y^2$ are invertible in $\mathfrak{S}^{-1}U(H)$. Thus
$S,T\in \mathfrak{S}^{-1}U(H)$. Moreover, following the notation of Proposition~\ref{prop:freealgebrainWeyl},
$$\Phi(S)=(t-\frac{5}{6})(t-\frac{1}{6})^{-1} +
(t-\frac{1}{6})(t-\frac{5}{6})^{-1}=s+s^{-1},$$
$$\Phi((z+y^2)^{-1}(z-y^2))=(1+p^2)^{-1}(1-p^2)=(1-p^2)(1+p^2)^{-1}=u,$$
$$\Phi((z+y^2)(z-y^2)^{-1})=(1+p^2)(1-p^2)^{-1}=u^{-1}.$$ Hence
$\Phi(T)=u(s+s^{-1})u^{-1}$.

By Proposition~\ref{prop:freealgebrainWeyl}, the set $\{s+s^{-1}, \
u(s+s^{-1})u^{-1}\}$ generates a free $k$-algebra. Therefore, the
$k$-algebra generated by $S$ and $T$ is a free $k$-algebra on the
set $\{S,\ T\}$.



It remains to show that $S$ and $T$ are symmetric with respect to
the principal involution $*$.

Clearly $V^*=V$. Thus $(V-\frac{1}{3}z^3)^*=(V+\frac{1}{3}z^3)$ and
$(V+\frac{1}{3}z^3)^*=V-\frac{1}{3}z^3$. Hence
\begin{eqnarray*}
S^* \ = \
\left((V-\frac{1}{3}z^3)(V+\frac{1}{3}z^3)^{-1} +
(V-\frac{1}{3}z^3)^{-1}(V+\frac{1}{3}z^3)\right)^*\ = \
S
\end{eqnarray*}
Now \begin{eqnarray*} \left((z+y^2)^{-1}(z-y^2)\right)^* & = &
(z-y^2)^*\left((z+y^2)^* \right)^{-1} \\
& = &  (-z-y^2)(-z+y^2)^{-1} \\
& = &  (z+y^2)(z-y^2)^{-1}.
\end{eqnarray*}
Thus
\begin{eqnarray*}
T^* & = & \left((z+y^2)^{-1}(z-y^2)S (z+y^2)(z-y^2)^{-1}\right)^*
\\ & = & \left((z+y^2)(z-y^2)^{-1}\right)^* S^*
\left((z+y^2)^{-1}(z-y^2)\right)^* \\ & = & (z+y^2)^{-1}(z-y^2)S
(z+y^2)(z-y^2)^{-1} \\ & = & T.
\end{eqnarray*}
Therefore $S$ and $T$ are symmetric with respect to the
principal involution.
\end{proof}

Simpler elements $S$ and $T$, symmetric with respect to the principal involution and that generate a free algebra in $\mathfrak{D}(H)$, can be found. The reason why we decided to pick
these more complicated elements is a technical one. It should become clearer at the end
of the proof of Theorem~\ref{theo:freesymmetricOre}.


\section{The case of a residually nilpotent Lie algebra}\label{sec:residuallynilpotent}

In this section, we develop a technique to obtain free algebras in $\mathcal{D}(L)$, where
$L$ is a (generalization of) a residually nilpotent Lie algebra from the ones obtained in $\mathfrak{D}(H)$, where $H$ is the Heisenberg Lie algebra.

\medskip

Let $L$  a Lie $k$-algebra generated by two elements $u,v$. Let
$H=\langle x,y\mid [[y,x],x]=[[y,x],y]=0\rangle$ be the Heisenberg
Lie $k$-algebra. Suppose that there exists a Lie $k$-algebra
homomorphism \begin{equation*}  L\stackrel{\rho}{\rightarrow} H,\
u\mapsto x,\ v\mapsto y.
\end{equation*}
Define $w=[v,u]$ and $z=[y,x]$. Let $N=\ker \rho$. Thus $N$ is a
(Lie) ideal of $L$.

By the universal property of  universal enveloping algebras, $\rho$
can be uniquely extended to a  $k$-algebra homomorphism $\psi\colon
U(L)\rightarrow U(H)$ between the corresponding universal enveloping
algebras. Note that $\ker \psi$ is the ideal of $U(L)$ generated by
$N$. The restriction $\psi_{|U(N)}$ coincides with the augmentation
map $\varepsilon\colon U(N)\rightarrow k$.

By the PBW-Theorem, the elements of $U(H)$ are uniquely expressed as
finite sums \linebreak $\sum_{l,m,n\geq 0} x^ly^mz^na_{lmn}$ with
$a_{lmn}\in k$. Let $\delta_x$ be the inner $k$-derivation of $U(H)$
determined by $x$, i.e. $\delta_x(f)=[f,x]=fx-xf$ for all $f\in
U(H)$. Notice that the $k$-subalgebra of $U(H)$ generated by $z$ can
be regarded as $k[z]$, the polynomial algebra in the variable $z$.
Since $[z,y]=0$, the $k$-subalgebra of $U(H)$ generated by $\{y,z\}$
can be viewed as $k[z][y]$. As $\delta_x(z)=[z,x]=0$,
$\delta_x(y)=[y,x]=z\in k[z][y]$, one can prove that
\begin{equation}\label{eq:U(L)aspolynomials}
U(H)=k[z][y][x;\delta_x]
\end{equation}
because (a), (b), (c), (d) in
Section~\ref{sec:differentialoperators} are easily verified.

Consider now  $U(L)$, the universal enveloping algebra of $L$. By
the PBW-Theorem, the elements of $U(L)$ can be uniquely expressed as
finite sums $\sum_{l,m,n\geq 0} u^lv^mw^nf_{lmn}$ with $f_{lmn}\in
U(N)$. Since $N$ is an ideal of $L$, the inner derivations
$\delta_u$, $\delta_v$, $\delta_w$ of $U(L)$ defined by $u,v,w$,
respectively, are such that $\delta_u(U(N))\subseteq U(N)$,
$\delta_v(U(N))\subseteq U(N)$, $\delta_w(U(N))\subseteq U(N)$. The
$k$-subalgebra of $U(L)$ generated by $U(N)$ and $w$ is
$U(N)[w;\delta_w]$ for properties (a), (b), (c), (d) in
Section~\ref{sec:differentialoperators} are easily verified. Since
$\delta_v(w)\in U(N)\subseteq U(N)[w;\delta_w]$, one can prove, as
before, that the $k$-subalgebra of $U(L)$ generated by $U(N)$ and
$\{w,v\}$ is $U(N)[w;\delta_w][v;\delta_v]$. Furthermore, since
$\delta_u(v)=w$ and $\delta_u(w)\in U(N)$,
\begin{equation}\label{eq:U(H)aspolynomials}
U(L)=U(N)[w;\delta_w][v;\delta_v][u;\delta_u].
\end{equation}

By \eqref{eq:U(L)aspolynomials}, $U(H)=k[z][y][x;\delta_x]$. By applying
Lemma~\ref{lem:extensiontocompletion}(2) twice, we get
\begin{equation}\label{eq:U(L)insideseries}
U(H)\hookrightarrow k((t_z))((t_y))((t_x;\delta_x)).
\end{equation}
More precisely, the first time, we obtain that $\delta_x$ can be
extended to $k((t_z))$. Thus obtaining the embedding
$U(H)\hookrightarrow k((t_z))[y][x;\delta_x]$. Applying
Lemma~\ref{lem:extensiontocompletion} a second time with the role of
$R$ played $k((t_x))[y]$ and the role of $w$ by $y$, we get that
$\delta_x$ can be extended to $k((t_z)((t_y))$. Hence we obtain the
embedding \eqref{eq:U(L)insideseries}.

By \eqref{eq:U(H)aspolynomials},
$U(L)=U(N)[w;\delta_w][v;\delta_v][u;\delta_u]$. Again,
Lemma~\ref{lem:extensiontocompletion}(2) yields an embedding
\begin{equation}\label{eq:U(H)insideseries}
U(L)\hookrightarrow
U(N)((t_w;\delta_w))((t_v;\delta_v))((t_u;\delta_u)).
\end{equation}
More precisely, applying  Lemma~\ref{lem:extensiontocompletion}(2)
twice, we get extensions of $\delta_v$ and $\delta_u$ to
$U(N)((t_w;\delta_w))$, resulting in an embedding
$U(L)\hookrightarrow
U(L)((t_w;\delta_w))[t_v;\delta_v][t_u;\delta_u]$. Again, by
Lemma~\ref{lem:extensiontocompletion}(2), we get an extension of
$\delta_u$ to $U(L)((t_w;\delta_w))((t_v;\delta_v))$ and,  hence, an
embedding of $U(L)$ into
$U(N)((t_w;\delta_w))((t_v;\delta_v))[t_u;\delta_u]$, and from this
we get \eqref{eq:U(H)insideseries}.

In this setting, we  prove the next results.

\begin{lem}\label{lem:commutativediagram}
There exists a commutative diagram of embeddings of $k$-algebras
\begin{equation*}\label{eq:diagramnotOre}
\xymatrixcolsep{0.0001cm}\xymatrix{U(L)=U(N)[w;\delta_w][v;\delta_v][u;\delta_u]\ar@{^{(}->}[rr]\ar@{^{(}->}[dd]
\ar@{^{(}->}[rd] & &
\mathfrak{D}(N)[w;\delta_w][v;\delta_v][u;\delta_u]\ar@{^{(}->}[dd]
\ar@{_{(}->}[ld]\\ &  \mathfrak{D}(L) \ar@{^{(}->}[dr] &\\
U(N)((t_w;\delta_w))((t_v;\delta_v))((t_u;\delta_u))\ar@{^{(}->}[rr]
& & \mathfrak{D}(N)((t_w;\delta_w))((t_v;\delta_v))((t_u;\delta_u))
& }
\end{equation*}
\end{lem}

\begin{proof}
By Proposition~\ref{prop:canonicalfieldoffractions}, the division
$k$-subalgebra of $\mathfrak{D}(L)$ generated by $U(N)$ is
$\mathfrak{D}(N)$ and the standard monomials on the set $\{u,v,w\}$
are $\mathfrak{D}(N)$-linearly independent.

Recall that if $f$ is an invertible element in a $k$-algebra  and
$\delta$ is a derivation, then
$\delta(f^{-1})=-f^{-1}\delta(f)f^{-1}$. Thus,  since the image of
$U(N)$ under the derivations $\delta_u$, $\delta_v$, $\delta_w$ is
contained in $U(N)$, one can prove that
$\delta_u(\mathfrak{D}(N))\subseteq \mathfrak{D}(N)$,
$\delta_v(\mathfrak{D}(N))\subseteq \mathfrak{D}(N)$ and
$\delta_w(\mathfrak{D}(N))\subseteq \mathfrak{D}(N)$. Hence we
obtain the embeddings
\begin{equation*}
\xymatrixcolsep{0.0001cm}\xymatrix{U(L)=U(N)[w;\delta_w][v;\delta_v][u;\delta_u]\ar@{^{(}->}[rr]\ar@{^{(}->}[dd]
\ar@{^{(}->}[rd] & &
\mathfrak{D}(N)[w;\delta_w][v;\delta_v][u;\delta_u]\ar@{^{(}->}[dd]
\ar@{_{(}->}[ld]\\ &  \mathfrak{D}(L)  &\\
U(N)((t_w;\delta_w))((t_v;\delta_v))((t_u;\delta_u))\ar@{^{(}->}[rr]
& & \mathfrak{D}(N)((t_w;\delta_w))((t_v;\delta_v))((t_u;\delta_u))
& }
\end{equation*}
The remaining embedding $\mathfrak{D}(L)\hookrightarrow
\mathfrak{D}(N)((t_w;\delta_w))((t_v;\delta_v))((t_u;\delta_u))$ is
obtained as follows. The division $k$-algebra $\mathfrak{D}(L)$ is
generated, as a division $k$-algebra, by $U(L)$. Hence
$\mathfrak{D}(L)$ is generated by
$\mathfrak{D}(N)[u;\delta_u][v;\delta_v][w;\delta_w]$. The
$k$-algebra $\mathfrak{D}(N)[w;\delta_w][v;\delta_v][u;\delta_u]$ is
an Ore domain (see for example \cite[Theorem~10.28]{Lam2}).
Therefore $\mathfrak{D}(L)$ is the Ore division ring of fractions of
$\mathfrak{D}(N)[w;\delta_w][v;\delta_v][u;\delta_u]$. The series
ring
$\mathfrak{D}(N)((t_w;\delta_w))((t_v;\delta_v))((t_u;\delta_u))$ is
a division $k$-algebra that contains
$\mathfrak{D}(N)[w;\delta_w][v;\delta_v][u;\delta_u]$; thus it
contains $\mathfrak{D}(L)$, as desired.
\end{proof}

\begin{lem} \label{lem:morphismsofseries}
Let $\varepsilon\colon U(N)\rightarrow k$ denote the augmentation
map. The following hold true.
\begin{enumerate}[\rm (1)]
\item There exists a  $k$-algebra homomorphism 
$$\Phi_w\colon U(N)((t_w;\delta_w))\rightarrow k((t_z)),\quad \sum_{i}t_w^if_i\mapsto \sum_{i}t_z^i\varepsilon(f_i),$$
where $f_i\in U(N)$ for each $i$.
\item There exists a $k$-algebra homomorphism 
$$\Phi_v\colon U(N)((t_w;\delta_w))((t_v;\delta_v))
\rightarrow k((t_z))((t_y)),\quad \sum_{i}t_v^ig_i\mapsto
\sum_{i}t_y^i\Phi_w(g_i),$$ where $g_i\in U(N)((t_w;\delta_w))$ for
each $i$.
\item There exists a $k$-algebra homomorphism
$$\Phi_u\colon U(N)((t_w;\delta_w))((t_v;\delta_v))((t_u;\delta_u))
\rightarrow k((t_z))((t_y))((t_x;\delta_x)),\quad
\sum_{i}t_u^ih_i\mapsto \sum_{i}t_x^i\Phi_v(h_i),$$ where $h_i\in
U(N)((t_w;\delta_w))((t_y;\delta_y))$ for each $i$.
\end{enumerate}
\end{lem}

\begin{proof}
Recall that $\ker\varepsilon$ is the ideal of $U(N)$ generated by
$N$. Also, any element $f\in U(N)$ is of the form $f=a+b$ with $a\in
k$ and $b\in \ker\varepsilon$. Since $N$ is an ideal of $H$,
$\delta_s(n)=[n,s]\in N$ for any $n\in N$, $s\in L$. Thus
\begin{equation}\label{eq:imageinsidekernel}
\delta_s(U(N))\subseteq \ker\varepsilon,\  \textrm{ for any } s\in
L.
\end{equation}

(1) The polynomial ring $k[z]$ can be regarded as $k[z;\delta_z]$
where $\delta_z=0$. By
Lemma~\ref{lem:extensionmorphismtocompletion}, it is enough to prove
that $\varepsilon(\delta_w(f))=\delta_z(\varepsilon(f))=0$ for all
$f\in U(N)$. This holds by \eqref{eq:imageinsidekernel}.

(2) The polynomial ring $k((t_z))[y]$ can be viewed as
$k((t_z))[y;\delta_y]$ where $\delta_y=0$. By
Lemma~\ref{lem:extensionmorphismtocompletion}, it is enough to prove
that
\begin{equation}\label{eq:imageinthekernel}
\Phi_w(\delta_v(g))=\delta_y(\Phi_w(g))=0 \textrm{ for all }
g\in U(N)((t_w;\delta_w)).
\end{equation}
Observe that it is enough to prove this
equality for $g\in U(N)[[t_w;\delta_w]]$. Let $g=\sum_{i\geq 0}
t_w^if_i$ with $f_i\in U(N)$ for all $i$. By
Lemma~\ref{lem:extensiontocompletion}(2),
$$\delta_v\Big(\sum_{i\geq 0}
t_w^if_i\Big)= \sum_{i\geq 0} \delta_v(t_w^if_i)=\sum_{i\geq 0}
\Big(t_w^i\delta_v(f_i)+\delta_v(t_w^i)f_i \Big).$$ By
\eqref{eq:imageinsidekernel}, $\delta_v(f_i)\in \ker\varepsilon$.
Therefore
$$\Phi_w\Big(\sum_{i\geq 0}
t_w^i\delta_v(f_i)\Big)=\sum_{i\geq 0}
t_z^i\varepsilon(\delta_v(f_i))=0.$$ By
Lemma~\ref{lem:extensiontocompletion}(2),
$$\delta_v(t_w)=-t_w\delta_v(w)t_w=\sum_{i=1}^\infty t_w^{i+1}\delta_w^{i-1}(\delta_v(w))(-1)^{i}.$$
Therefore $\delta_v(t_w)$ is a series with coefficients in
$\ker\varepsilon$. Suppose that $\delta_v(t_w^{i-1})$ is a series
with coefficients in $\ker\varepsilon$. Then
$$\delta_v(t_w^{i})=\delta_v(t_w^{i-1}t_w)=t_w^{i-1}\delta_v(t_w)+\delta_v(t_w^{i-1})t_w.$$
From this, it is not difficult to show that  $\delta_v(t_w^i)$ is a series with coefficients in
$\ker\varepsilon$. Hence $\sum_{i\geq 0} \delta_v(t_w^i)f_i$ is a
series with coefficients in $\ker\varepsilon$, and  $\Phi_w\Big(
\sum_{i\geq 0} \delta_v(t_w^i)f_i \Big)=0$. Therefore
$\Phi_w(\delta_w(g))=0$, as desired.

(3) By Lemma~\ref{lem:extensionmorphismtocompletion}, it is enough
to prove that
$$\Phi_v(\delta_u(h))=\delta_x(\Phi_v(h)) \textrm{ for all }
h\in U(N)((t_w;\delta_w))((t_v;\delta_v)).$$ Observe that it is
enough to prove this equality for $h\in
U(N)((t_w;\delta_w))[[t_v;\delta_v]]$. Let $h=\sum_{i\geq 0}
t_v^ig_i$ where $g_i\in U(N)((t_w;\delta_w))$ for all $i$. By
Lemma~\ref{lem:extensiontocompletion}(2),
$$\delta_u\Big(\sum_{i\geq 0} t_v^ig_i \Big)=\sum_{i\geq0}
\Big(t_v^i\delta_u(g_i)+\delta_u(t_v^i)g_i \Big).$$ In
\eqref{eq:imageinthekernel}, we proved that $\delta_v(g_i)\in \ker
\Phi_w$. In the same way, one can show that $\delta_u(g_i)\in \ker
\Phi_w$. Thus $\Phi_v\Big(\sum_{i\geq 0}
t_v^i\delta_u(g_i)\Big)=\sum_{i\geq 0}
t_y^i\Phi_w(\delta_u(g_i))=0$. Hence we only have to worry about the
term $\sum_{i\geq 0}\delta_u(t_v^i)g_i$. We prove, by induction on
$i$, that $\delta_u(t_v^i)=t_v^{i+1}w(-i) + A_i$, with $A_i\in \ker
\Phi_v$. For $i=0$, the result is clear. For $i=1$,
$\delta_u(t_v)=-t_v\delta_u(v)t_v=t_vwt_v(-1)$. By
Lemma~\ref{lem:extensiontocompletion}(2),
$$wt_v=\sum_{i\geq 1}t_v^i\delta_v^{i-1}(w)(-1)^{i-1}=t_vw +A'_1,$$ where $A'_1=\sum_{i\geq2}t_v^i\delta_v^{i-1}(w)(-1)^{i-1}$,
 a series in $t_v$ where the  coefficient of  $t^i_v$ is
$\delta_v^{i-1}(w)(-1)^{i-1}\in \ker\varepsilon$ for $i>1$. Thus
$\delta_u(t_v)=t_v^2w(-1) + t_vA'_1(-1)$. If we set
$A_1=t_vA_1'(-1)$, the result is proved for $i=1$. Suppose the
result holds for $l<i$. Then
\begin{eqnarray*}
\delta_u(t_v^i) & = & \delta_u(t_v^{i-1}t_v) \\
& = & t_v^{i-1}\delta_u(t_v)+\delta_u(t_v^{i-1})t_v\\
& = & t_v^{i-1}(t_v^2w(-1)+A_1)+ (t_v^iw(-(i-1))+A_{i-1})t_v\\
& = &
t_v^{i+1}w(-1)+t_v^{i+1}w(-(i-1))+t_v^{i-1}A_1+A_{i-1}t_v+t_v^iA_1'(-(i-1))\\
& = & t_v^{i+1}w(-i) + t_v^{i-1}A_1+A_{i-1}t_v+t_v^iA_1'(-(i-1)).
\end{eqnarray*}
Now it is not difficult to prove that $A_i=
t_v^{i-1}A_1+A_{i-1}t_v+t_v^iA_1'(-(i-1))$ is a series in $t_v$ with
coefficients in $\ker\varepsilon$. Thus $A_i\in\ker \Phi_v$.

Hence, we have, on the one hand, \begin{eqnarray*}
 \Phi_v\Big(\delta_u\Big(\sum_{i\geq 0} t_v^ig_i\Big)\Big) & = &  \Phi_v\Big(\sum_{i\geq 0}
\Big(t_v^i\delta_u(g_i)+\delta_u(t_v^i)g_i \Big)   \Big)  \\
& = & \Phi_v\Big(\sum_{i\geq 0} t_v^i\delta_u(g_i)\Big) +
\Phi_v \Big(\sum_{i\geq 0} \delta_u(t_v^i)g_i \Big) \\
& = & \sum_{i\geq 0} t_y^i\Phi_w(\delta_u(g_i)) +
\sum_{i\geq 0} \Phi_v(t_v^{i+1}w(-i) + A_i)\Phi_w(g_i) \\
& = & \sum_{i\geq 0} t_y^{i+1}\Phi_w(w(-i))\Phi_w(g_i) \\
& = & \sum_{i\geq 0} t_y^{i+1}z\Phi_w(g_i)(-i).
\end{eqnarray*}
On the other hand,
\begin{eqnarray*}
\delta_x\Big(\Phi_v\Big(\sum_{i\geq 0} t_v^ig_i \Big)\Big) & =
& \delta_x\Big( \sum_{i\geq 0} t_y^i\Phi_w(g_i)  \Big) \\
& = & \sum_{i\geq 0} \Big( t_y^i\delta_x(\Phi_w(g_i)) +
\delta_x(t_y^i)\Phi_w(g_i) \Big). \\
\end{eqnarray*}
First observe that $\delta_x(\Phi_w(g_i))=0$, because $\Phi_w(g_i)$
is a series in $t_z$ with coefficients in $k$, and $z$ commutes with
$x$. If we prove, by induction on $i$, that
$\delta_x(t_y^i)=t_y^{i+1}z(-i)$, the result is proved. For $i=0$,
this is clear. For $i=1$,
$$\delta_x(t_y)=-t_y\delta_x(y)t_y=-t_yzt_y=t_y^2z(-1),$$
where we have used that $z$ and $y$ commute (and hence $z$ and
$t_y$). Suppose now that the identity holds for $l<i$. Then
\begin{eqnarray*}
\delta_x(t_y^i) & = & \delta_x(t_y^{i-1}t_y) \\
& = & t_y^{i-1}\delta_x(t_y) + \delta_x(t_y^{i-1})t_y \\
& = & t_y^{i-1}t_y^2(-z) + (t_y^iz(-(i-1)))t_y \\
& = & t_y^{i+1}z(-1) + t_y^{i+1}z(-(i-1))=t_y^{i+1}z(-i),
\end{eqnarray*}
as desired.
\end{proof}

Now we are ready to prove the main result of this section.

\begin{theo}\label{theo:freesymmetricresiduallynilpotent}
Let $k$ be a field of  characteristic zero, let $H=\langle x,y\mid
[[y,x],x]=[[y,x],y]=0\rangle$ be the Heisenberg Lie $k$-algebra and
let $L$  be a Lie $k$-algebra generated by two elements $u,v$.  Suppose
that there exists a Lie $k$-algebra homomorphism \begin{equation*}
L{\rightarrow} H,\ u\mapsto x,\ v\mapsto y.
\end{equation*}
Let $w=[v,u]$, $V=\frac{1}{2}w(uv+vu)w$, and consider the following
elements of $\mathfrak{D}(L)$:
$$S=(V-\frac{1}{3}w^3)(V+\frac{1}{3}w^3)^{-1} +
(V-\frac{1}{3}w^3)^{-1}(V+\frac{1}{3}w^3),$$
$$T=(w+v^2)^{-1}(w-v^2)S(w+v^2)(w-v^2)^{-1}.$$ The following hold
true.
\begin{enumerate}[\rm(1)]
\item The $k$-subalgebra of $\mathfrak{D}(L)$
generated by $S$ and $T$ is the free $k$-algebra on the set
$\{S,T\}$.
\item The elements $S$ and $T$ are symmetric with respect to the
principal involution on $\mathfrak{D}(L)$.
\end{enumerate}
\end{theo}

\begin{proof}
Define $z=[y,x]\in H$. Consider the embedding
$U(H)\hookrightarrow k((t_z))((t_y))((t_x;\delta_x))$
given in \eqref{eq:U(L)insideseries}. Since $k((t_z))((t_y))((t_x;\delta_x))$ is a division $k$-algebra and $U(H)$
is an Ore domain, it extends to an embedding
$\mathfrak{D}(H)\hookrightarrow k((t_z))((t_y))((t_x;\delta_x))$.

Let $N=\ker (L\rightarrow H)$ and consider  the embedding
$U(L)\hookrightarrow
U(N)((t_w;\delta_w))((t_v;\delta_v))((t_u;\delta_u))$ given in
\eqref{eq:U(H)insideseries}. Let $\Phi_u\colon
U(N)((t_w;\delta_w))((t_v;\delta_v))((t_u;\delta_u)) \rightarrow
k((t_z))((t_y))((t_x;\delta_x))$  be the homomorphism given in
Lemma~\ref{lem:morphismsofseries}.

Define the following elements in $\mathfrak{D}(H)$:
$V_H=\frac{1}{2}z(xy+yx)z$,
$$S_H=(V_H-\frac{1}{3}z^3)(V_H+\frac{1}{3}z^3)^{-1} +
(V_H-\frac{1}{3}z^3)^{-1}(V_H+\frac{1}{3}z^3),$$
$$T_H=(z+y^2)^{-1}(z-y^2)S_H(z+y^2)(z-y^2)^{-1}.$$
By
Theorem~\ref{theo:freesymmetricHeisenberg},  the $k$-algebra generated by $S_H$ and $T_H$ is the free $k$-algebra on the set $\{S_H,T_H\}$.

We claim that the elements $V-\frac{1}{3}w^3$, $V+\frac{1}{3}w^3$,
$w+v^2$ and $w-v^2$ are all invertible in
$U(N)((t_w;\delta_w))((t_v;\delta_v))((t_u;\delta_u))$. If the claim
is true, we will have $\Phi_u(V)=V_H$, $\Phi_u(S)=S_H$ and
$\Phi_u(T)=T_H$. Moreover, by Lemma~\ref{lem:commutativediagram},
$V$, $S$ and $T$ belong to $\mathfrak{D}(L)$. Therefore, the
elements $S$ and $T$ are nonzero and invertible in
$\mathfrak{D}(L)$, and the $k$-algebra generated by them is the free
$k$-algebra on the set $\{S,T\}$.

We proceed to prove the claim. We begin with the element
$w+v^2=t_w^{-1}+t_v^{-2}$. As a series in $t_v$, this element is
invertible in $U(N)((t_w;\delta_w))((t_v;\delta_v))$ if and only if
the coefficient of $t_v^{-2}$ is invertible in the ring of
coefficients  $U(N)((t_w;\delta_w))$. The coefficient is $1$, which
is clearly invertible. Similarly, it can be proved that $w-v^2$ is
invertible. Now we show that $V+\frac{1}{3}w^3$ is invertible in
$U(N)((t_w;\delta_w))((t_v;\delta_v))((t_u;\delta_u))$. First we
obtain an expression of $V+\frac{1}{3}w^3$ as a series in $t_u$.
\begin{eqnarray*}
V+\frac{1}{3}w^3 & = & \frac{1}{2} w(uv+vu)w +\frac{1}{3}w^3 \\ & =
& \frac{1}{2}w(uv+[v,u] +uv)w + \frac{1}{3}w^3 \\ & = &
\frac{1}{2}w^3+  wuvw + \frac{1}{3}w^3 \\ & = & \frac{5}{6}w^3 +wuvw
\\ & = & \frac{5}{6}t_w^{-3} +wuvw.
\end{eqnarray*}
Now,
\begin{eqnarray*}
wvuw & = & (uw+[w,u])vw \\ & = & u(vw+[w,v])w + (v[w,u]+[[w,u],v])w
\\  & = & uvw^2+ u(w[w,v]+[[w,v],w])+v(w[w,u]+[[w,u],w])+w[[w,u],v]+[[[w,u],v],w]  \\
& = & t_u^{-1}t_v^{-1}t_w^{-2}+ t_u^{-1}t_w^{-1}[w,v] +
t_u^{-1}[[w,v],w]
 + \\ & &  t_v^{-1}t_w^{-1}[w,u] + t_v^{-1}[[w,u],w]+
t_w^{-1}[[w,u],v]+ [[[w,u],v],w]
  \\
& = &  t_u^{-1}(t_v^{-1}t_w^{-2}+ t_w^{-1}[w,v] + [[w,v],w])+
\\ & &  t_v^{-1}t_w^{-1}[w,u] + t_v^{-1}[[w,u],w]+
t_w^{-1}[[w,u],v]+ [[[w,u],v],w].
\end{eqnarray*}
Thus, as a series in $t_u$, the coefficient of the least element in
the support  of $V+\frac{1}{3}w^3$ is $t_v^{-1}t_w^{-2}+
t_w^{-1}[w,v] + [[w,v],w]\in U(N)((t_w;\delta_w))((t_v;\delta_v))$.
Hence $V+\frac{1}{3}w^3$ is invertible in
$U(N)((t_w;\delta_w))((t_v;\delta_v))((t_u;\delta_u))$ if and only
if $t_v^{-1}t_w^{-2}+ t_w^{-1}[w,v] + [[w,v],w]$ is invertible in
$U(N)((t_w;\delta_w))((t_v;\delta_v))$. But, as a series in $t_v$
with coefficients in $U(N)((t_w;\delta_w))$, the coefficient of the
least element in the support of $t_v^{-1}t_w^{-2}+ t_w^{-1}[w,v] +
[[w,v],w]$ is $t_w^{-2}$. Clearly, $t_w^{-2}$ is invertible in
$U(N)((t_w;\delta_w))$. This implies that $t_v^{-1}t_w^{-2}+
t_w^{-1}[w,v] + [[w,v],w]$ is invertible in
$U(N)((t_w;\delta_w))((t_v;\delta_v))$. Therefore $V+\frac{1}{3}w^3$
is invertible in
$U(N)((t_w;\delta_w))((t_v;\delta_v))((t_u;\delta_u))$. The case of
$V-\frac{1}{3}w^3$ is shown analogously, and the claim is proved.

We have just proved (1). To prove that $S$ and $T$ are symmetric
with respect to the principal involution, one proceeds as in the
proof of Theorem~\ref{theo:freesymmetricHeisenberg} substituting
$u,v,w$ for $x,y,z$, respectively.
\end{proof}

\medskip

Let $K$ be a Lie $k$-algebra. By a \emph{central series} of $K$ we
mean a set $\mathcal{K}$ of ideals of $K$, called the \emph{terms}
of $K$, which is linearly ordered by inclusion,  contains $0$ and
$K$, and $\mathcal{K}$ also satisfies the following conditions:
\begin{enumerate}[(i)]
\item If $0\neq x\in K$, there are terms of $\mathcal{K}$ which do
not contain $x$ and the union of all such terms is a term $V_x$ of
$\mathcal{K}$.
\item If $0\neq x\in K$, there are terms of $\mathcal{K}$ which
contain $x$ and the intersection of all such terms is a term
$\Lambda_x$ of $\mathcal{K}$.
\item Each factor $\Lambda_x/V_x$ is central in $K$, that is, $[K,\Lambda_x]\subseteq
V_x$.
\end{enumerate}

\begin{ex}\label{ex:centralseries}
Let $K$ be a Lie $k$-algebra.
\begin{enumerate}[(a)]
\item Suppose that $K$ is nilpotent, that is, there exists a finite
sequence of ideals of $K$
$$K=K_1\supseteq K_2\supseteq \dotsb\supseteq K_n=0,$$
such that $[K,K_i]\subseteq K_{i+1}$. Then
$\mathcal{K}=\{K_i\}_{i=1}^n$ is a central series of $K$.
\item More generally, suppose that $K$ is a \emph{residually nilpotent} Lie
$k$-algebra, that is, there exists a sequence $\{K_i\}_{i=1}^\infty$
of ideals of $K$ such that $K_1=K$, $K_i\supseteq K_{i+1}$,
$\bigcap_{i=1}^\infty K_i=0$ and $[K,K_i]\subseteq K_{i+1}$. Then
$\mathcal{K}=\{0\}\cup\{K_i\}_{i=1}^\infty$ is a central series of $K$.

\item Even more generally, suppose that $K$ is a \emph{hypo-central} Lie algebra, that is,
$K$ is a Lie $k$-algebra with a sequence of ideals
$\{K_\lambda\}_{\lambda<\varepsilon}$, where $\varepsilon$ is an
ordinal, such that $K_1=K,$ $K_\lambda\supseteq K_{\lambda+1}$,
$[K,K_\lambda]\subseteq K_{\lambda+1}$,
$K_\lambda=\bigcap_{\beta<\lambda}K_\beta$ when $\lambda$ is a limit
ordinal, and $\bigcap_{\lambda<\varepsilon} K_\lambda=0$. Then
$\mathcal{K}=\{0\}\cup\{K_\lambda\}_{\lambda<\varepsilon}$ is a central
series.

\item Let $M$ be a semigroup with an order relation such that, for
all $x,y,z\in M$, we have
\begin{enumerate}[(i)]
\item  $x<y$ implies  $zx<zy$ and $xz<yz$, and
\item $x<xy$, $x<yx$.
\end{enumerate}
Suppose that a Lie $k$-algebra $K$ has a family of $k$-vector
subspaces $\{K^x\}_{x\in M}$ such that $K=\bigoplus_{x\in M} K^x$
and $[K^x,K^y]\subseteq K^{xy}$. Define ideals of $K$ by
$V_x=\bigcup_{y\geq x} K^y$ and $\Lambda_x=\bigcup_{y>x} K^y$ (it
may happen that $V_x=\Lambda_y$ for some $x,y$). Then
$\mathcal{K}=\{V_x,\Lambda_x\mid x\in M\}\cup {K}\cup\{0\}$
(possibly not disjoint) is a central series.

\item A Lie $k$-algebra $K$ is \emph{hypercentral} if there exists a chain
of ideals $\{K_\mu\}_{\mu\leq\nu}$ of $K$ (indexed by some ordinal
$\nu$) that satisfies the following conditions: $K_0=0$, $K_\nu=K$,
 $K_\mu\subseteq K_{\mu+1}$ for all $0\leq \mu<\nu$,
 $K_{\mu'}=\bigcup_{\mu<\mu'}K_\mu$ for all limit ordinals
$\mu'\leq \nu$, $[K,K_{\mu+1}]\subseteq K_\mu$ for all $\mu<\nu$, or
equivalently, $K_{\mu+1}/K_\mu$ is contained in the center of
$K/K_{\mu}$. Then $\mathcal{K}=\{K_\mu\}_{\mu\leq\nu}$ is a central
series.

\item There exists a concept of orderable Lie algebra, defined in
\cite{KopytovOrderedLieAlgebras}. It is proved in
\cite[Corollary~3.5]{KopytovOrderedLieAlgebras} that every orderable
Lie $k$-algebra admits a central series.  \qed

\end{enumerate}

\end{ex}

\begin{coro}\label{coro:freesymmetricresiduallynilpotent}
Let $k$ be a field of  characteristic zero. Let $K$ be a Lie
$k$-algebra with a central series $\mathcal{K}$.

Let $u,v\in K$ be such that $[v,u]\neq 0$. Denote by $L$ the Lie
$k$-subalgebra of $K$ generated by $\{u,v\}$. Then there exists an
ideal $N$ of $L$ such that $L/N$ is isomorphic to $H$, the
Heisenberg Lie $k$-algebra. Therefore, if we define $w=[v,u]$,
$V=\frac{1}{2}w(uv+vu)w$,
$$S=(V-\frac{1}{3}w^3)(V+\frac{1}{3}w^3)^{-1} +
(V-\frac{1}{3}w^3)^{-1}(V+\frac{1}{3}w^3),$$
$$T=(w+v^2)^{-1}(w-v^2)S(w+v^2)(w-v^2)^{-1},$$ the following hold
true.
\begin{enumerate}[\rm(1)]
\item The $k$-subalgebra of $\mathfrak{D}(K)$ generated by $S$ and $T$ is the free $k$-algebra on the set $\{S,T\}$.
\item The elements $S$ and $T$ are symmetric with respect to the
principal involution on $\mathfrak{D}(K)$.
\end{enumerate}
\end{coro}

\begin{proof}
For $x\in L$, let $V_x$ be the union of all  terms  of $\mathcal{K}$
that do not contain $x$, and let $\Lambda_x$ be the intersection of
all the terms of $\mathcal{K}$ which contain $x$.

Notice that $\Lambda_u\subseteq \Lambda_v$ or $\Lambda_v\subseteq
\Lambda_u$. Let $t\in\{u,v\}$ be such that
$\Lambda_t=\Lambda_u\cap\Lambda_v$. Consider $\Lambda_t$ and $V_t$.

Since $[K,\Lambda_t]\subseteq V_t$, the element $w=[v,u]\in V_t$. By
definition, $w\in \Lambda_w\setminus V_w$. Since $\Lambda_w\subseteq
V_t$, then $u,v\notin \Lambda_w$. Define $N=V_w\cap L$. Since
$[v,u]=w\notin V_w$, then $w\notin N$, and thus $L/N$ is not
commutative. Observe that $[L,L]\subseteq \Lambda_w$. Therefore
$[L,[L,L]]\subseteq [L,L\cap\Lambda_w]\subseteq L\cap V_w=N$. Hence
$L/N$ is a noncommutative 3-dimensional Lie $k$-algebra with basis
$\{\bar{u},\bar{v},\bar{w}\}$,  the classes of $u$, $v$ and $w$ in
$L/N$. Moreover $[L/N,L/N]=k\bar{w}$ which is contained in the
center of $L/N$. Therefore $L/N$ is the Heisenberg Lie $k$-algebra.

By Theorem~\ref{theo:freesymmetricresiduallynilpotent}, the result
holds for $\mathfrak{D}(L)$. Since
$\mathfrak{D}(L)\hookrightarrow\mathfrak{D}(K)$ by
Proposition~\ref{prop:canonicalfieldoffractions}(1), the result follows.
\end{proof}

The first part of proof of
Corollary~\ref{coro:freesymmetricresiduallynilpotent} consists on
showing that there exists an ideal $N$ of $L$ such that $L/N$ is
isomorphic to the Heisenberg Lie $k$-algebra. This argument is
analogous to the one used  for ordered groups in \cite[after
Proposition~3.4]{SanchezfreegroupalgebraMNseries}.


\section{The universal enveloping algebra is Ore}\label{sec:U(L)isOre}

We begin this section gathering technical results from \cite{Lichtmanfreeuniversalenveloping}.

\begin{prop}\label{prop:specializationfromgraduation}
Let $R$ be an Ore domain with a  valuation $\chi\colon R\rightarrow
\mathbb{Z}\cup\{\infty\}$. This valuation induces a filtration $$
\dotsb \supseteq R_{-1}\supseteq R_0\supseteq R_1\supseteq\dotsb $$
with $R_i=\{a\in R\mid \chi (a)\geq i\}$. Let $R[t,t^{-1}]$ be the
Laurent polynomial ring with coefficients in $R$.

\begin{enumerate}[\rm (1)]
\item The valuation $\chi$ can be extended to
$\chi\colon R[t,t^{-1}]\rightarrow \mathbb{Z}\cup\{0\}$ by
defining $$\chi \left(\sum_{i=l}^n t^ia_i\right)=\min_i
\{\chi(a_i)+i\},$$ for all $\sum_{i=l}^n t^ia_i\in R[t,t^{-1}]$.

\item Let $T=\{x\in R[t,t^{-1}]\mid \chi(x)\geq 0\}$ and
$T_0=\{x\in R[t,t^{-1}]\mid \chi(x)>0\}$. Then there exists a
ring isomorphism
$$\begin{array}{rcl} \gr(R) & \stackrel{\varphi}{\rightarrow} & T/T_0 \\
a\in R_i/R_{i+1} & \mapsto & t^{-i}a+T_0   \end{array}$$

\item $T$ and $\gr(R)$ are Ore domains.

\item $\mathcal{S}=T\setminus T_0$ is a left denominator set, $\mathcal{S}^{-1}T$
is a local ring with maximal ideal $\mathcal{S}^{-1}T_0$.

\item Let $D$ and $\Delta$ be the Ore division rings
of fractions of $T$ and $\gr(R)$, respectively. The ring
homomorphism $$T\rightarrow
T/T_0\stackrel{\varphi^{-1}}{\rightarrow}\gr(R)\hookrightarrow
\Delta$$ can be extended to a ring homomorphism $\psi\colon
\mathcal{S}^{-1}T\rightarrow \Delta$ which induces an isomorphism
$$\mathcal{S}^{-1}T/\mathcal{S}^{-1}T_0\cong \Delta.$$ This
isomorphism extends
$T/T_0\stackrel{\varphi^{-1}}{\rightarrow}\gr(R)$.
\end{enumerate}
\end{prop}
\begin{proof}
(1) is proved in \cite[p.~87]{Cohnskew}.

(2) is shown in \cite[p.~87]{Cohnskew}.

(3) is proved in \cite[Propositions~17 and
18(i)]{Lichtmanfreeuniversalenveloping}.

(4) is shown in \cite[Proposition~17]{Lichtmanfreeuniversalenveloping}.

(5) is proved in
\cite[Proposition~18(ii)]{Lichtmanfreeuniversalenveloping}.
\end{proof}

Now we are ready to prove the main result of this section. The technique used is from \cite{Lichtmanfreeuniversalenveloping}.
\begin{theo}\label{theo:freesymmetricOre}
Let $k$ be a field of  characteristic zero. Let $L$ be a Lie
$k$-algebra such that its universal enveloping algebra $U(L)$ is an
Ore domain. Let $\mathfrak{D}(L)$ be its Ore division ring of
fractions. Let $u,v\in L$ such that the Lie subalgebra generated by
them is of dimension at least three.

Define $w=[v,u]$, $V=\frac{1}{2}w(uv+vu)w$, and consider the
following elements of $\mathfrak{D}(L)$:
$$S=(V-\frac{1}{3}w^3)(V+\frac{1}{3}w^3)^{-1} +
(V-\frac{1}{3}w^3)^{-1}(V+\frac{1}{3}w^3),$$
$$T=(w+v^2)^{-1}(w-v^2)S(w+v^2)(w-v^2)^{-1}.$$ The following hold
true.
\begin{enumerate}[\rm(1)]
\item The $k$-subalgebra of $\mathfrak{D}(L)$ generated by $S$ and $T$ is the free  $k$-algebra
on the set $\{S,T\}$.
\item The elements $S$ and $T$ are symmetric with respect to the
principal involution on $\mathfrak{D}(L)$.
\end{enumerate}
\end{theo}

\begin{proof}
Let $L_1$ be the Lie subalgebra of $L$ generated by $u$ and $v$.

Since $U(L)$ is an Ore domain, $U(L_1)$ is also an Ore domain. Moreover,
$\mathfrak{D}(L_1)\subseteq \mathfrak{D}(L)$. Thus, we may suppose that $L$ is generated by $u$ and $v$.

For $l\leq 0$, let $U_l(L)$ be the $k$-subspace of $U(L)$ spanned by all the monomials in $u$ and $v$ of length $\leq -l$.
In this way, we obtain  a filtration
\begin{equation}\label{eq:filtrationofU(H)}
\dotsb \supseteq U_{l}(L)\supseteq U_{l+1}(L)\supseteq \dotsb
\supseteq U_{-1}(L) \supseteq U_0(L)=k\supseteq 0=U_1(L)
\end{equation}
of $U(L)$. Let $\gr(U(L))$ be the graded ring associated to this filtration: $$\gr(U(L))=
\bigoplus_{l\leq 0} U_{l}(L)/U_{l+1}(L).$$
Let $L_{l}=U_l(L)\cap L$, $l\leq 0$, be the induced filtration in $L$. We obtain in the natural way the
graded Lie $k$-algebra $\gr(L)$ associated to this filtration. It can be shown that
$\gr(L)$ generates $\gr(U(L))$ as a $k$-algebra and that
\begin{equation}\label{eq:isomorphismofgraded}
U(\gr(L))\cong\gr(U(L))
\end{equation}
under a natural isomorphism, i.e. the one extending
$\gr(L)\rightarrow \gr(U(L))$
\cite[Proposition~14]{Lichtmanfreeuniversalenveloping} or
\cite[Lemma~2.1.2]{Boiscorpsenveloppants}. In what follows,  the two
objects in \eqref{eq:isomorphismofgraded} will be identified.

As a first consequence of \eqref{eq:isomorphismofgraded}, we get
that $\gr(U(L))$ is a domain. Therefore the filtration
\eqref{eq:filtrationofU(H)} induces a valuation $\chi\colon
U(L)\rightarrow \mathbb{Z}\cup\{\infty\}$ by $\chi(x)=l$ if $x\in
U_l(L)\setminus U_{l+1}(L)$ and $\chi(0)=\infty$
\cite[Proposition~2.6.1]{Cohnskew}.

Now we can apply the results of
Proposition~\ref{prop:specializationfromgraduation} to $R=U(L)$.
Consider the Laurent polynomial ring $U(L)[t,t^{-1}]$, and extend
$\chi$ to a valuation, also denoted by $\chi$,
$U(L)[t,t^{-1}]\rightarrow \integers \cup\{\infty\}$, defined, as in
Proposition~\ref{prop:specializationfromgraduation}(1), by
$$\chi \left(\sum_{i=l}^n t^ia_i\right)=\min_i
\{\chi(a_i)+i\},$$ for all $\sum_{i=l}^n t^ia_i\in
U(L)[t,t^{-1}]$.

Let $T=\{x\in U(L)[t,t^{-1}]\mid \chi(x)\geq 0\}$ and
$T_0=\{x\in U(L)[t,t^{-1}]\mid \chi(x)>0\}$. Then there
exists a ring isomorphism $\varphi\colon \gr(U(L))\rightarrow T/T_0$
by Proposition~\ref{prop:specializationfromgraduation}(2).

Consider $u,v$ and $w=[v,u]$. Note that $u,v\in U_{-1}(L)$ and $w\in
U_{-2}(L)\setminus U_{-1}(L)$ because $L$ is not two-dimensional. In
other words, $\chi(u)=\chi(v)=-1$ and $\chi(w)=-2$. Thus, also
$u,v\in L_{-1}$ and $w\in L_{-2}\setminus L_{-1}$. Denote by
$\bar{u}$, $\bar{v}$ the class of $u,v\in U_{-1}(L)/U_0(L)$ and also
the class of $u$ and $v$ in $L_{-1}/L_0$. Denote by $\bar{w}$ the
class of $w$ in $U_{-2}(L)/U_{-1}(L)$ and in $L_{-2}/L_{-1}$. Then
$\varphi(\bar{u})=tu+T_0$ and $\varphi(\bar{v})=tv+T_0$ and
$\varphi(\bar{w})=t^2w+T_0$. By
Proposition~\ref{prop:specializationfromgraduation}, $T$ and
$U(\gr(L))$ are Ore domains. Let $D$ be the Ore division ring of
fractions of $T$ and let $\Delta$ be the Ore division ring of
fractions of $U(\gr(L))$.

Now, $\gr(L)$ is a (negatively) graded Lie $k$-algebra which is not
commutative $(w\in L_{-2}\setminus L_{-1})$. Thus $\gr(L)$ has to be
a residually nilpotent Lie $k$-algebra. Observe that
$[\bar{v},\bar{u}]=\bar{w}$ as elements of $\gr(L)$.

Now define
$\overline{V}=\frac{1}{2}\bar{w}(\bar{u}\bar{v}+\bar{v}\bar{u})\bar{w}$,
$$\overline{S}=(\overline{V}-\frac{1}{3}\bar{w}^3)(V+\frac{1}{3}\bar{w}^3)^{-1} +
(\overline{V}-\frac{1}{3}\bar{w}^3)^{-1}(\overline{V}+\frac{1}{3}\bar{w}^3),$$
$$\overline{T}=(\bar{w}+\bar{v}^2)^{-1}(\bar{w}-\bar{v}^2)\overline{S}(\bar{w}+\bar{v}^2)(\bar{w}-\bar{v}^2)^{-1}.$$
Then Corollary~\ref{coro:freesymmetricresiduallynilpotent}(1) proves
that $\overline{S}$ and $\overline{T}$ generate a free $k$-algebra
in $\Delta$.

Let $\mathcal{S}=T\setminus T_0$. By
Proposition~\ref{prop:specializationfromgraduation}(4),
$\mathcal{S}$ is a left denominator set of $T$, and
$\mathcal{S}^{-1}T$ is a local ring with maximal ideal
$\mathcal{S}^{-1}T_0$. Moreover, there exists a morphism of rings
$\psi\colon \mathcal{S}^{-1}T\rightarrow \Delta$ which induces the
isomorphism $\mathcal{S}^{-1}T/\mathcal{S}^{-1}T_0\cong \Delta$, by
Proposition~\ref{prop:specializationfromgraduation}(5).

Consider $tu$, $tv$, $t^2w\in T\setminus T_0$. Observe that
$\psi(tu)=\bar{u}$, $\psi(tv)=\bar{v}$ and $\psi(t^2w)=\bar{w}$.

Define $V'=\frac{1}{2}wt^2(utvt+vtut)wt^2=t^6V$. Since $u,v,w$ are
$k$-linearly independent in $L$, the PBW-Theorem implies that $u$, $v$ and $uv+vu=w+2uv$ are $k$-linearly
independent in $U(L)$. Hence $uv+vu\in U_{-2}(L)\setminus U_{-1}(L)$
and $\chi(uv+vu)=-2$. Therefore
$\chi(V)=\chi(w(uv+vu)w)=-6.$  By the definition of
$\chi$,  $\chi(V')=0$.  Hence $V'\in T\setminus T_0$.
In the same way, one can show that
$V'-\frac{1}{3}(t^2w)^3=t^6(V-\frac{1}{3}w^3)\in T\setminus T_0$,
$V'+\frac{1}{3}(t^2w)^3=t^6(V+\frac{1}{3}w^3)\in T\setminus T_0 $,
$t^2w+(tv)^2=t^2(w+v^2)\in T\setminus T_0$ and
$t^2w-(tv)^2=t^2(w-v^2)\in T\setminus T_0$. Now define two elements
$S'$, $T'$ of $\mathcal{S}^{-1}T$ as follows:
$$S'=(V'-\frac{1}{3}(t^2w)^3)(V'+\frac{1}{3}(t^2w)^3)^{-1} +
(V'-\frac{1}{3}(t^2w)^3)^{-1}(V'+\frac{1}{3}(t^2w)^3),$$
$$T'=(t^2w+(tv)^2)^{-1}(t^2w-(tv)^2)S'(t^2w+(tv)^2)(t^2w-(tv)^2)^{-1}.$$
Observe that $\psi(S')=\overline{S}$ and $\psi(T')=\overline{T}$.
Therefore the $k$-algebra generated by $\{S',T'\}$ is the free
$k$-algebra on $\{S',T'\}$.

And here comes magic. Since $t$ is an element of the center, it
follows that $T'=T$ and $S'=S$. Thus (1) is proved.

For (2), as in the proof of
Theorem~\ref{theo:freesymmetricHeisenberg}, it is not difficult to
prove that $S$ and $T$ are symmetric with respect to the principal
involution.
\end{proof}

\begin{rem}
Let $k$ be a field of characteristic zero. Let $L$ be a Lie
$k$-algebra generated by two nonconmuting elements $u$ and $v$ such
that $U(L)$  is an Ore domain. Suppose that $L$ is of dimension at
least three. By Lichtman's technique used in
Theorem~\ref{theo:freesymmetricOre}, the problem of finding a free
$k$-algebra (generated by symmetric elements or not) inside
$\mathfrak{D}(L)$ is reduced to the problem of finding a free Lie
$k$-algebra inside  $\mathfrak{D}(\gr(L))$, where $\gr(L)$ is a
positively graded Lie $k$-algebra, $\gr(L)=\bigoplus_{i=1}^\infty
\gr_i(L)$, which is not commutative and  is generated by $\gr_1(L)$.
Then, by Corollary~\ref{coro:freesymmetricresiduallynilpotent},
$\mathfrak{D}(\gr(L))$ contains a free $k$-algebra because it is a
residually nilpotent Lie $k$-algebra

There is another way of obtaining a free  $k$-algebra inside
$\mathfrak{D}(L)$ as follows. We have the isomorphism
$\frac{\gr(L)}{\bigoplus_{i=3}^\infty\gr_i(L)}\cong H$, the
Heisenberg Lie $k$-algebra. By
\cite[Proposition~7.7]{Lichtmanvaluationmethodsgrouprings} and
\cite[Corollary~7.2]{Lichtmanvaluationmethodsgrouprings}, there
exists a specialization from $\mathfrak{D}(\gr(L))$ to
$\mathfrak{D}(H)$ extending the homomorphism of Lie $k$-algebras
$\gr(L)\rightarrow
\frac{\gr(L)}{\bigoplus_{i=3}^\infty\gr_i(L)}\cong H$. Thus a
suitable preimage of the free  $k$-algebra obtained in
Theorem~\ref{theo:freesymmetricHeisenberg} does the work.  We would
like to remark that our method of obtaining the free  $k$-algebra
inside $\mathfrak{D}(\gr(L))$ is more elementary since
\cite[Proposition~7.7]{Lichtmanvaluationmethodsgrouprings} and
\cite[Corollary~7.2]{Lichtmanvaluationmethodsgrouprings} rely on
highly nontrivial facts. Moreover, our method is more general
because it works for all classes of Lie $k$-algebras appearing in
Examples~\ref{ex:centralseries},  while this other method only works
for positively graded Lie $k$-algebras (as $\gr(L)$). \qed
\end{rem}

When the Lie subalgebra generated by $u$ and $v$ is of dimension two, we cannot apply the methods developed thus far, but we have the following
consequence of Cauchon's Theorem.

\begin{prop}\label{prop:twodimensionalcase}
Let $k$ be a field of characteristic zero. Let $M$ be the
noncommutative two dimensional Lie $k$-algebra. Thus $M$ has a basis
$\{e,f\}$ such that $[e,f]=f$. Define
$s=(e-\frac{1}{3})(e+\frac{1}{3})^{-1}$ and $u=(1-f)(1+f)^{-1}$.
Consider the embedding $U(M)\hookrightarrow\mathfrak{D}(M)$. Then
\begin{enumerate}[\rm(1)]
\item the elements $s+s^{-1}$ and $u(s+s^{-1})u^{-1}$ are symmetric
with respect to the principal involution;
\item the $k$-algebra generated by $\{s+s^{-1},\ u(s+s^{-1})u^{-1}\}$ is
the free $k$-algebra on $\{s+s^{-1},\ u(s+s^{-1})u^{-1}\}$.
\end{enumerate}
\end{prop}

\begin{proof}
Since $[e,f]=ef-fe=f$, $ef=f(e+1)$. Thus  $U(M)$ can be seen as a
skew polynomial $k$-algebra, $U(M)=k[e][f;\sigma]$, where
$\sigma(e)=e+1$.

According to Cauchon's Theorem, if we define
$s=(e-\frac{1}{3})(e+\frac{1}{3})^{-1}$ and $u=(1-f)(1+f)^{-1}$, the
elements $s$ and $usu^{-1}$ generate a free group $k$-algebra. Then
Corollary~\ref{coro:freeinsidegroupring} implies that the
$k$-algebra generated by $\{s+s^{-1},\ u(s+s^{-1})u^{-1}\}$ is the
free $k$-algebra on $\{s+s^{-1},\ u(s+s^{-1})u^{-1}\}$. Moreover
$$s^*=\left((e-\frac{1}{3})(e+\frac{1}{3})^{-1}\right)^*=
(-e+\frac{1}{3})^{-1}(-e-\frac{1}{3})=(e+\frac{1}{3})(e-\frac{1}{3})^{-1}=s^{-1},$$
$$u^*=\left((1-f)(1+f)^{-1}\right)^*=(1-f)^{-1}(1+f)=(1+f)(1-f)^{-1}=u^{-1}.$$
Therefore $s+s^{-1}$ and $u(s+s^{-1})u^{-1}$, are symmetric.
\end{proof}


\section{Further comments}\label{sec:furthercomments}

Let $k$ be a field of characteristic zero, $L$ be a noncommutative Lie $k$-algebra and $U(L)$ be its universal enveloping algebra.

Until now, we have proved that if $u,v\in L$ with $[u,v]\neq 0$ and
$K$ is the Lie $k$-subalgebra they generate, then
$\mathfrak{D}(K)\left(\subseteq\mathfrak{D}(L)\right)$  contains a
free $k$-algebra generated by symmetric elements with respect to the
principal involution of $\mathfrak{D}(L)$ and  we give explicit
generators of this free $k$-algebra, provided that either $U(K)$ is
an Ore domain or  there exists a morphism of Lie $k$-algebras
$K\rightarrow H$, $u\mapsto x$, $v\mapsto y$, where $H$ is the
Heisenberg Lie $k$-algebra.

Moreover, if $L$ is a Lie $k$-algebra that contains two elements
$u,v$ such that the Lie $k$-algebra generated by them is the free
Lie $k$-algebra on the set $\{u,v\}$, then $\mathfrak{D}(L)$
contains a free $k$-algebra generated by symmetric elements with
respect to the principal involution. Indeed, let $F$ be the free Lie
$k$-algebra generated by $\{u,v\}$. Recall that $U(F)$, the
universal enveloping algebra of $F$, is the free $k$-algebra
$\freealgebra k {u,v}$. Hence $u^2$ and $v^2$ are symmetric, and the
$k$-algebra they generate is the free $k$-algebra $\freealgebra
k{u^2,v^2}\subseteq U(F)\subseteq U(L)\subseteq \mathfrak{D}(L)$.
(Note that, in this case, the characteristic of $k$ need not be
zero.)

Now we raise the following natural question in the case of a field $k$ of characteristic zero. What do possible counterexamples of Lie $k$-algebras $L$ such that $\mathfrak{D}(L)$ does not contain a free noncommutative $k$-algebra generated by symmetric elements look like?

We assert that such a possible counterexample cannot be of
subexponential growth. Indeed, by \cite{Smithsubexponentialgrowth},
if $L$ is a Lie $k$-algebra  of subexponential growth, then $U(L)$
is of subexponential growth. Then, since a noncommutative free
$k$-algebra is of exponential growth, we obtain that $U(L)$ does not
contain a noncommutative free $k$-algebra and, therefore, $U(L)$ is
an Ore domain because any $k$-algebra which is not an Ore domain
must contain a noncommutative free $k$-algebra
\cite[Proposition~10.25]{Lam2}. Now recall that if $L_1$ is a Lie
subalgebra of $L$, then $U(L_1)$ is an Ore domain (because $U(L)$
is). Thus the claim is proved.

Therefore $L$ must be of exponential growth but  cannot contain a
noncommutative free Lie $k$-algebra. Also, $L$ must satisfy the
following properties: for any two noncommuting elements $u,v\in L$,
if we denote by $K$ the Lie $k$-subalgebra they generate,  then $K$
cannot be of subexponential growth, and $K$ cannot be mapped onto
$H$ (in particular $K$ can neither be residually nilpotent nor have
a central series). Although we think that such examples could exist,
we have not been able to find any.

\bigskip

The next theorem is the main result  of
\cite{Lichtmanfreeuniversalenveloping}. We show that the results
obtained thus far also prove it.

\begin{theo}
Let $k$ be a field of characteristic zero and $L$ be a
noncommutative Lie $k$-algebra. Then any division ring that contains
the universal enveloping algebra $U(L)$ of $L$ contains a
noncommutative free $k$-algebra.
\end{theo}

\begin{proof}
It is well knwon that if  $U(L)$  is not  Ore, then it contains a noncommutative free $k$-algebra \cite[Propostion~10.25]{Lam2}.

If $U(L)$ is an Ore domain, then $\mathfrak{D}(L)$ is the Ore division ring of fractions of $U(L)$ and it is contained in any division ring that contains $U(L)$. Then Theorem~\ref{theo:freesymmetricOre} and Propostion~\ref{prop:twodimensionalcase} imply that $\mathfrak{D}(L)$ contains a free $k$-algebra.
\end{proof}

\medskip

Let now $k$ be an uncountable field. Suppose that
$D$ is a division ring that contains $k$ as a central subfield. Suppose that there exist different elements $S$ and $T$ in $D$ that generate a free $k$-algebra $k\langle S,T\rangle$. It was shown in \cite{GoncalvesShirvani} that there exist $a,b\in k$ such that the $k$-algebra generated by $\{1+aS,\ (1+aS)^{-1},\ 1+bT ,\ (1+bT)^{-1}\}$
is the  free group $k$-algebra  on the set $\{1+aS,\ 1+bT\}$.

Let now $L$ be a Lie $k$-algebra with $k$ uncountable and of
characteristic zero.  Suppose that $L$ is one of the Lie
$k$-algebras for which we have proved that $\mathfrak{D}(L)$
contains two symmetric elements with respect to the principal
involution $S$ and $T$ such that the $k$-algebra they generate is
free on $\{S,T\}$. By the foregoing paragraph, there exist $a,b\in
k$ such that the $k$-algebra generated by $\{1+aS,\ (1+aS)^{-1},\
1+bT ,\ (1+bT)^{-1}\}$ is the free group $k$-algebra  on the set
$\{1+aS,\ 1+bT\}$. Observe that $1+aS$ and $1+bT$ are also symmetric
with respect to the principal involution.

\bibliographystyle{amsplain}
\bibliography{grupitosbuenos}

\end{document}